\documentclass[10pt,twoside,a4paper,english,reqno]{amsart}
\usepackage{amscd}
\usepackage{caption}
\usepackage{a4wide}
\usepackage{amssymb,amsthm,amsmath,amsfonts,mathrsfs}
\usepackage{bm,latexsym,enumitem,dsfont,tabularx,graphicx,url}
\usepackage[utf8]{inputenc}
\usepackage{hyperref}
\usepackage{stmaryrd}
\hypersetup{
    colorlinks=true,
    linkcolor=blue,
    filecolor=red,
    urlcolor=blue,
    citecolor=red,
}
\usepackage{nicefrac,microtype,upref,ulem,doi,footmisc}

\theoremstyle{plain}
\newtheorem{theorem}{Theorem}[section]

\newtheorem{lemma}[theorem]{Lemma}

\newtheorem{assumption}[theorem]{Assumption}

\numberwithin{equation}{section}

\usepackage{tikz,tikzscale}
\usepackage{pgfplots}
\usepackage{subcaption}
\pgfplotsset{compat=1.13}

\newcommand{\logLogSlopeTriangle}[5]
{

    \pgfplotsextra
    {
        \pgfkeysgetvalue{/pgfplots/xmin}{\xmin}
        \pgfkeysgetvalue{/pgfplots/xmax}{\xmax}
        \pgfkeysgetvalue{/pgfplots/ymin}{\ymin}
        \pgfkeysgetvalue{/pgfplots/ymax}{\ymax}

        \pgfmathsetmacro{\xArel}{#1}
        \pgfmathsetmacro{\yArel}{#3}
        \pgfmathsetmacro{\xBrel}{#1-#2}
        \pgfmathsetmacro{\yBrel}{\yArel}
        \pgfmathsetmacro{\xCrel}{\xArel}

        \pgfmathsetmacro{\lnxB}{\xmin*(1-(#1-#2))+\xmax*(#1-#2)} 
        \pgfmathsetmacro{\lnxA}{\xmin*(1-#1)+\xmax*#1} 
        \pgfmathsetmacro{\lnyA}{\ymin*(1-#3)+\ymax*#3} 
        \pgfmathsetmacro{\lnyC}{\lnyA+#4*(\lnxA-\lnxB)}
        \pgfmathsetmacro{\yCrel}{\lnyC-\ymin)/(\ymax-\ymin)} 

        \coordinate (A) at (rel axis cs:\xArel,\yArel);
        \coordinate (B) at (rel axis cs:\xBrel,\yBrel);
        \coordinate (C) at (rel axis cs:\xCrel,\yCrel);

        \draw[#5]   (A)-- node[pos=0.5,anchor=south] {1}
                    (B)-- 
                    (C)-- node[pos=0.5,anchor=west] {#4}
                    cycle;
    }
}

\usepackage{xcolor}
\definecolor{lateksii_color}{RGB}{155,0,119}

 \title[HDG for the non-local CH-KP equation]{A Hybridizable Discontinuous Galerkin Method for the non--local Camassa--Holm--Kadomtsev--Petviashvili equation}

\author{Mukul Dwivedi \and Ruben Gutendorf \and Andreas Rupp}
\address{Department of Mathematics, Saarland University, Saarbrücken, Germany}
\email{\{mukul.dwivedi;ruben.gutendorf;andreas.rupp\}@uni-saarland.de}
\thanks{This work has been supported by the Deutsche Forschungsgemeinschaft (DFG, German
Research Foundation) -- 577175348.}
\begin{document}

\begin{abstract}
This paper develops a hybridizable discontinuous Galerkin method for the two-dimensional Camassa--Holm--Kadomtsev--Petviashvili equation. The method employs Cartesian meshes with tensor-product polynomial spaces, enabling separate treatment of \(x\) and \(y\) derivatives. The non-local operator \(\partial_{x}^{-1}u_{y}\) is localized through an auxiliary variable \(v\) satisfying \(v_x = u_y\), allowing efficient element-by-element computations. We prove energy stability of the semi-discrete scheme and derive \(\mathcal{O}(h^{k+1/2})\) convergence in space. Numerical experiments validate the theoretical results and demonstrate the method's capability to accurately resolve smooth solutions and peaked solitary waves (peakons).
\end{abstract}

\maketitle
\section{Introduction}
The Camassa-Holm (CH) equation is the task to find the fluid velocity $u:\mathbb R\times \mathbb R^+ \to \mathbb R$ such that
\begin{equation}\label{eq:ch}
u_t - u_{txx} + 2\kappa u_x + 3uu_x = 2u_xu_{xx} + uu_{xxx}, \quad \text{ for }~ x \in \mathbb{R}, \; t > 0,
\end{equation}
where $\kappa > 0$ represents the background water depth. It constitutes a fundamental model for unidirectional shallow water waves over flat topography \cite{camassa1994new,CamassaHolm1993,Constantin1998,Johnson2002}, and peaked solitary wave (peakon) solutions of the form \(u(x,t) = ce^{-|x-ct|}\), for $c\in \mathbb R$, if \(\kappa = 0\).
The existence of these peaked solutions is intimately connected to essential nonlinear phenomena such as finite-time wave breaking \cite{ConstantinEscher1998}, which the equation successfully models \cite{camassa1994new, rodriguez2001cauchy}. The mathematical structure of \eqref{eq:ch} has been extensively analyzed, encompassing spectral stability \cite{ConstantinStrauss2002}, orbital stability \cite{constantin2001orbital} of the smooth solitary waves, and peakon instability \cite{constantin2000stability,NataliPelinovsky2020} of the peaked traveling waves. Global well-posedness and blow up results for \eqref{eq:ch} were studied in \cite{ConstantinEscher1998,Constantin1998}. For recent developments in the well-posedness of the CH equation \eqref{eq:ch}, we refer to \cite{brandolese2014local} and references therein.

Importantly, the CH equation \eqref{eq:ch} represents the evolving free surface with a single, one-dimensional (1D) time-dependent surface, thereby restricting the fluid motion to a vertical plane.  This simplification suppresses the weak transverse variations that occur in realistic oceanographic settings. Transverse modulations of the free surface are naturally captured by letting the displacement depend on the two spatial variables $(x,y)\in\mathbb{R}^{2}$, leading to a two-dimensional (2D) generalization, the Camassa–Holm–Kadomtsev–Petviashvili (CH–KP) equation to find $u:\mathbb R^2\times \mathbb R^+ \to \mathbb R$ such that
\begin{equation}\label{eq:chkp}
\bigl(u_{t}-u_{t xx}+2\kappa u_{x}+3uu_{x}-2u_{x}u_{x x}-uu_{x x x}\bigr)_{x}+u_{y y}=0,
\qquad \text{for }~ (x,y)\in\mathbb{R}^{2},\; t>0,
\end{equation}
which was first derived in the context of nonlinear elasticity \cite{Chen2006} and later rigorously justified for irrotational shallow-water flows \cite{GuiLiuLuoYin2021}.  Equation~\eqref{eq:chkp} extends the CH framework in the spirit of the Kadomtsev–Petviashvili (KP) theory \cite{KadomtsevPetviashvili1970}, while retaining the stronger nonlinearity of CH, and thus provides a richer alternative to the classical KP \cite{KadomtsevPetviashvili1970} and Korteweg--de Vries \cite{Craig1985} approximations.  Its applicability ranges from pure shallow-water dynamics to regimes influenced by Coriolis forcing \cite{GuiLiuSun2019}.  The transverse stability of smooth solitary waves and global well-posedness of the CH-KP equation \eqref{eq:chkp} in Sobolev spaces has been studied in \cite{Pelinovsky2024,GuiLiuLuoYin2021}.
 
 The nonlocal evolution form emerges through the application of the nonlocal operator \( \partial_x^{-1} \) yielding
\begin{equation}\label{eq:evo}
u_t + (1 - \partial_x^2)^{-1} \left[  f(u)_x - 2u_xu_{xx} - uu_{xxx} + \partial_x^{-1}u_{yy} \right] = 0,
\end{equation}
where \( f(u) = 2\kappa u + \frac{3}{2}u^2 \), and \( \partial_x^{-1}g(x) := \int_{+\infty}^x g(s)  ds \) is defined for functions decaying to 0 as \( x \to +\infty \), see \cite{Pelinovsky2024}.
This formulation admits a Hamiltonian structure
\begin{equation}\label{eq:energy}
u_t = -J F'(u) \quad \text{ with }\quad  
F[u] := \frac{1}{2} \iint_{\mathbb{R}^2} \left( u^3 + uu_x^2  + (\partial_x^{-1}u_y)^2\right)\, \mathrm{dx dy},
\end{equation} 
where \( J := \partial_x(1-\partial_x^2)^{-1}\) is a skew-adjoint operator and \( F \) represents a conserved energy for local solutions \( u \in X^s(\mathbb R^2) := \{ u \in H^s(\mathbb R^2) : \partial_x^{-1}u \in H^s(\mathbb R^2),~ u_x \in H^s(\mathbb R^2) \}\) with \(s \geq 2\). This system possesses additional conserved quantities, including momentum
\begin{equation}\label{eq:momentum}
E[u] = \frac{1}{2} \iint_{\mathbb{R}^2} \left( u^2 + u_x^2 \right)\, \mathrm{dx dy},
\end{equation}
and mass \cite{Pelinovsky2024,GuiLiuLuoYin2021}. The conservation of \( E[u] \) follows from multiplying \eqref{eq:evo} by \( (1-\partial_x^2)u \) and integrating over \( \mathbb R^2 \).
The evolution form \eqref{eq:evo} provides a viable framework for analysis by embedding dispersive effects within the Helmholtz resolvent $(1-\partial_x^2)^{-1}$, which is coercive on $L^2(\Omega)$ and facilitates energy-based approaches \cite{Pelinovsky2024}. Various wave breaking criteria for the equation \eqref{eq:evo}
were obtained in \cite[Theo. 1.2-1.4]{GuiLiuLuoYin2021}.

For numerical purposes, we consider the initial-boundary value problem on a sufficiently large bounded domain $\Omega := (x_L, x_R) \times (y_B, y_T)\subset \mathbb R^2$ and for a given $T>0$, i.e.,
\begin{subequations}\label{eq:ibvp}
\begin{align}
u_t + (1 - \partial_x^2)^{-1} \left[ f(u)_x - (uu_x)_{xx} + \frac{1}{2}(u_x^2)_{x} + \partial_x^{-1} u_{yy} \right] &= 0,\qquad (x,y)\in \Omega, \quad t\in(0,T), \\
u_0(x,y)&=u(x,y,0), \qquad (x,y)\in \Omega, \\
u(x_L,y,t) = u(x_R,y,t)   &= 0, \qquad y\in (y_B,y_T), \quad t\in(0,T),\\
 u(x,y_B,t) &= 0, \qquad x\in (x_L,x_R), ~~ ~t\in(0,T),\\
(\partial_x^{-1} u_y)(x_R,y,t) &= 0, \qquad y\in (y_B,y_T), \quad t\in (0,T),\\
(\partial_x^{-1} u_y)(x,y_T,t) &= 0, \qquad x\in (x_L,x_R), \quad t\in (0,T),\\
\partial_x u(x_L,y,t) = \partial_x u(x_R,y,t) &= 0, \qquad y\in (y_B,y_T), \quad t\in (0,T).
\end{align}
\end{subequations}
These boundary conditions ensure: (i) compatibility of $\partial_x^{-1}$ by enforcing flux balance at $x_R$ \cite{ZhangXuLiu2023}, (ii) decay properties mimicking the whole-space problem \cite{GuiLiuLuoYin2021}, and (iii) consistency with the Helmholtz structure. It should be noted that the above set of boundary conditions is not unique; other admissible choices (e.g., periodic or different homogeneous/inhomogeneous data) can also be employed, provided they are compatible with the well‑posedness theory. Local well-posedness holds for initial data $u_0 \in X^s(\mathbb R^2)$ with $s \geq 2$ \cite[Theo. 1.1]{GuiLiuLuoYin2021}.

Numerical approaches for 1D CH-type equations  \eqref{eq:ch} encompass diverse methodologies, each with distinct strengths and limitations.  Finite difference schemes \cite{coclite2008convergent,galtung2021numerical,HoldenRaynaud2006} utilize adaptive stencils to achieve convergence, error estimates for Galerkin method are derived in \cite{antonopoulos2019error}, and the finite volume method \cite{ArtebrantSchroll2006} incorporates shock-capturing techniques for adaptive peakon resolution while maintaining mass conservation. Spectral techniques \cite{dwivedi2025convergent,KalischLenells2005,Kalisch2006Xavier} leverage exponential convergence for smooth periodic solutions.  Operator splitting methods \cite{ZhanZhou2018} decouple nonlinear advection from dispersive terms via Strang splitting. The local discontinuous Galerkin (LDG) formulation for 1D CH equation pioneered by Xu and Shu \cite{XuShu2008}, which establishes energy stability through auxiliary flux variables, and this framework is extended to the 2D CH-KP system on Cartesian grids by Zhang, Xu, and Liu in  \cite{ZhangXuLiu2023}. It establishes $\mathcal{O}(h^k)$ order of convergence in space. Related LDG approaches, including the invariant-preserving DG scheme \cite{liu2016invariant} for Hamiltonian CH dynamics and the $\mu$-CH solvers in one dimension are examined in \cite{zhang2019local}. For more DG methods, developed for the similar nonlinear dispersive equations and nonlinear nonlocal dispersive equations, we refer to \cite{bona2013conservative,dwivedi2024analysis,levy2004local,yan2002local,ZhangXuShu2021} and references therein. These methods duplicate interface unknowns, leading to expanded global degrees of freedom and auxiliary variables, and  $\mathcal{O}(h^k)$ or $\mathcal{O}(h^{k+{\frac{1}{2}}})$ convergence is attained. 
These considerations motivate the hybridizable discontinuous Galerkin (HDG) methodology developed in what follows.

HDG methods, introduced by Cockburn et al. \cite{Cockburn2008super,CockburnGopalakrishnanLazarov2009,CockburnGopalakrishnanSayas2010}, constitute a significant advancement in (high-order) finite element discretizations. By reformulating discontinuous Galerkin frameworks through hybridization, HDG methods decompose the solution process into two complementary components: element-local solvers that compute approximations to the unknown in the bulk given trace values on the mesh skeleton, and a global system that enforces continuity through numerical fluxes. The methodology has been systematically extended beyond its original elliptic focus. Foundational work established HDG formulations for parabolic systems \cite{nguyen2009implicit,NguyenPeraireCockburn2009}, incompressible Navier--Stokes equations \cite{nguyen2011implicit}, phase-field models including Cahn-Hilliard equations \cite{chen2023superconvergence,kirk2024numerical}, with recent advances including multigrid solvers \cite{MR4803816}. For nonlinear dispersive wave equations, specialized HDG implementations have achieved notable success such as energy-stable schemes for KdV equations \cite{ChenCockburnDong2016,ChenDongJiang2018,Dong2017,Samii2016}, well-balanced explicit HDG discretizations of Serre–Green–Naghdi systems \cite{samii2018explicit}, for strongly Damped Wave Problem \cite{gupta2025hybridizable}, and for equations in continuum mechanics \cite{nguyen2012hybridizable}.

Central to these developments is the design of the trace operators on the interfaces, typically employing stabilized fluxes of the form
\begin{equation}
\hat{\omega}_h = \omega_h + \tau ( \widehat u_h -u_h),
\end{equation}
where $\tau$ is a stabilization parameter that ensures the well-posedness of local problems while maintaining accuracy and stability. Nevertheless, the nonlinear CH equation \eqref{eq:ch} and the nonlocal nonlinear CH-KP equation \eqref{eq:ibvp} present distinctive computational challenges for HDG frameworks due to three interconnected factors: the presence of peakon solutions with weak derivative continuity requirements \cite{CamassaHolm1993,Constantin1998}, nonlocal operator structures that complicate flux conservation \cite{GuiLiuLuoYin2021}, and higher order nonlinearities. While HDG has demonstrated success with related dispersive models \cite{ChenCockburnDong2016,Dong2017}, its application to peakon-supporting systems remains a largely unexplored research topic requiring novel methodological adaptations.

In this work, we develop the first HDG method for the CH-KP equation \eqref{eq:ibvp}, employing tensor-product polynomial spaces $Q_k(K_{ij}) = P_k(I_i) \otimes P_k(J_j)$, where $K_{ij} = [x_{i-1}, x_i] \times [y_{j-1}, y_j]=:I_i\times J_j$, to decouple spatial derivatives ({\it cf.} Section \ref{sec:notation}). The nonlocal operator $\partial_x^{-1}$ presents unique challenges since the CH-KP equation \eqref{eq:ibvp} describes the time evolution of $\partial_x u$ rather than $u$ itself, requiring additional constraints for well-posedness \cite{wang1994wave}. To address this, we implement a reconstruction technique for the primitive, $v_h$ which approximates $v:= \partial_x^{-1}u_y$, by rewriting $v_x = u_y$ and imposing a boundary constraint   $\widehat{v}_{h}(x_R,y) = 0$ at the right domain boundary $\Gamma_R$ ({\it cf.} Section \ref{sec:notation}), which anchors the reconstruction while preserving uniqueness, following the {\it non-integration DG} framework established for related one dimensional nonlocal equation \cite{zhang2020discontinuous} involving operator $\partial^{-1}_x$. Global coupling is minimized via three trace variables on the mesh skeleton $\mathcal{F}_h$: $\widehat{u}_h$, $\widehat{q}_h$, and the reconstructed primitive $\widehat{v}_h$, where $\widehat{u}_h$ and $\widehat{q}_h$ are approximations of the trace of the exact solution $u$ and $u_x$ on the interfaces, respectively. Compared to multi-trace LDG methods, this architecture reduces degrees of freedom while preserving stability and improving convergence order. We establish $\mathcal{O}(h^{k+1/2})$ order of convergence in the energy norm for $k \geq 2$ under the \textit{a priori} assumptions:
\begin{equation}
\|u-u_h\|_{\mathcal{T}_h} +\|q-q_h\|_{\mathcal{T}_h} \leq h^{2}.
\end{equation}

In this paper, we denote $C$ as a generic constant whose value can change in each step and it is independent of $h$.

The remainder of the paper is organized as follows. 
Section~\ref{sec:notation} introduces the notation and defines the tensor-product function spaces for the HDG discretization. 
Section~\ref{sec:formulation} details the HDG formulation for the CH-KP equation and establishes semi-discrete energy stability.  Section~\ref{sec:error} provides a $\mathcal{O}(h^{k+1/2})$ order of convergence of the scheme under appropriate regularity assumptions.  Section~\ref{sec:numerics} presents numerical experiments validating theoretical convergence rates and demonstrating peakon-resolution capabilities. 
Concluding remarks and future research directions are summarized in Section~\ref{sec:conclusion}. 
Technical proofs of supporting lemmas and validation of assumptions are deferred to Appendix~\ref{sec:proofs}.

\section{Notation and tensor--product spaces}\label{sec:notation}
Throughout, we work on the axis-aligned rectangle $\Omega = (x_L,x_R)\times(y_B,y_T)\subset\mathbb{R}^2$. The left and right vertical boundaries of $\Omega$ are denoted by $\Gamma_L$ and $\Gamma_R$, respectively, and the top and bottom horizontal boundaries of $\Omega$ are denoted by $\Gamma_T$ and $\Gamma_B$, respectively.
Let the one-dimensional cell intervals be
\[
I_i := (x_{i-1},x_i), \quad
J_j := (y_{j-1},y_j), \quad
1\le i\le N_x,\; 1\le j\le N_y,
\]
with $x_0=x_L$, $x_{N_x}=x_R$, $y_0=y_B$, $y_{N_y}=y_T$. 
The rectangular elements are
\[
K_{ij} := I_i \times J_j, \quad
1\le i\le N_x,\; 1\le j\le N_y,
\]
constituting the Cartesian mesh
\[
\mathcal{T}_h := \bigl\{K_{ij} \mid 1\le i\le N_x,\; 1\le j\le N_y \bigr\}, \quad
h := \max_{i,j} \bigl\{ |I_i|, |J_j| \bigr\}.
\]
 The sets of vertical faces and horizontal faces are
\[
\mathcal{V} := \bigl\{ \mathcal{V}_{i,j} := \{x = x_i\} \times J_j \mid 0 \le i \le N_x,\; 1 \le j \le N_y \bigr\}, 
\]
\[
\mathcal{H} := \bigl\{ \mathcal{H}_{i,j} := I_i \times \{y = y_j\} \mid 1 \le i \le N_x,\; 0 \le j \le N_y \bigr\}. 
\]
The full set of faces is $\mathcal{F}_h := \mathcal{V} \cup \mathcal{H}$, with boundary faces $\mathcal{F}_h^\partial := \mathcal{F}_h \cap \partial\Omega$. For each element $K_{ij}$, its boundary consists of four oriented face segments, and we write the set as follows
\[
\partial K_{ij} := \bigl\{ \Gamma_{i,j}^L,\; \Gamma_{i,j}^R,\; \Gamma_{i,j}^B,\; \Gamma_{i,j}^T \bigr\},
\]
 displayed in Figure \ref{fig:cartesian_grid}. The outward normal vector $\mathbf{n} = (n_x, n_y)$ is fixed for each face type: $\mathbf{n} = (-1,0)$ on left vertical faces $\Gamma_{i,j}^L$, $\mathbf{n} = (1,0)$ on right vertical faces $\Gamma_{i,j}^R$, $\mathbf{n} = (0,-1)$ on bottom horizontal faces $\Gamma_{i,j}^B$, and $\mathbf{n} = (0,1)$ on top horizontal faces $\Gamma_{i,j}^T$.
\begin{figure}[htbp]
\centering
\includegraphics[width=0.5\textwidth]{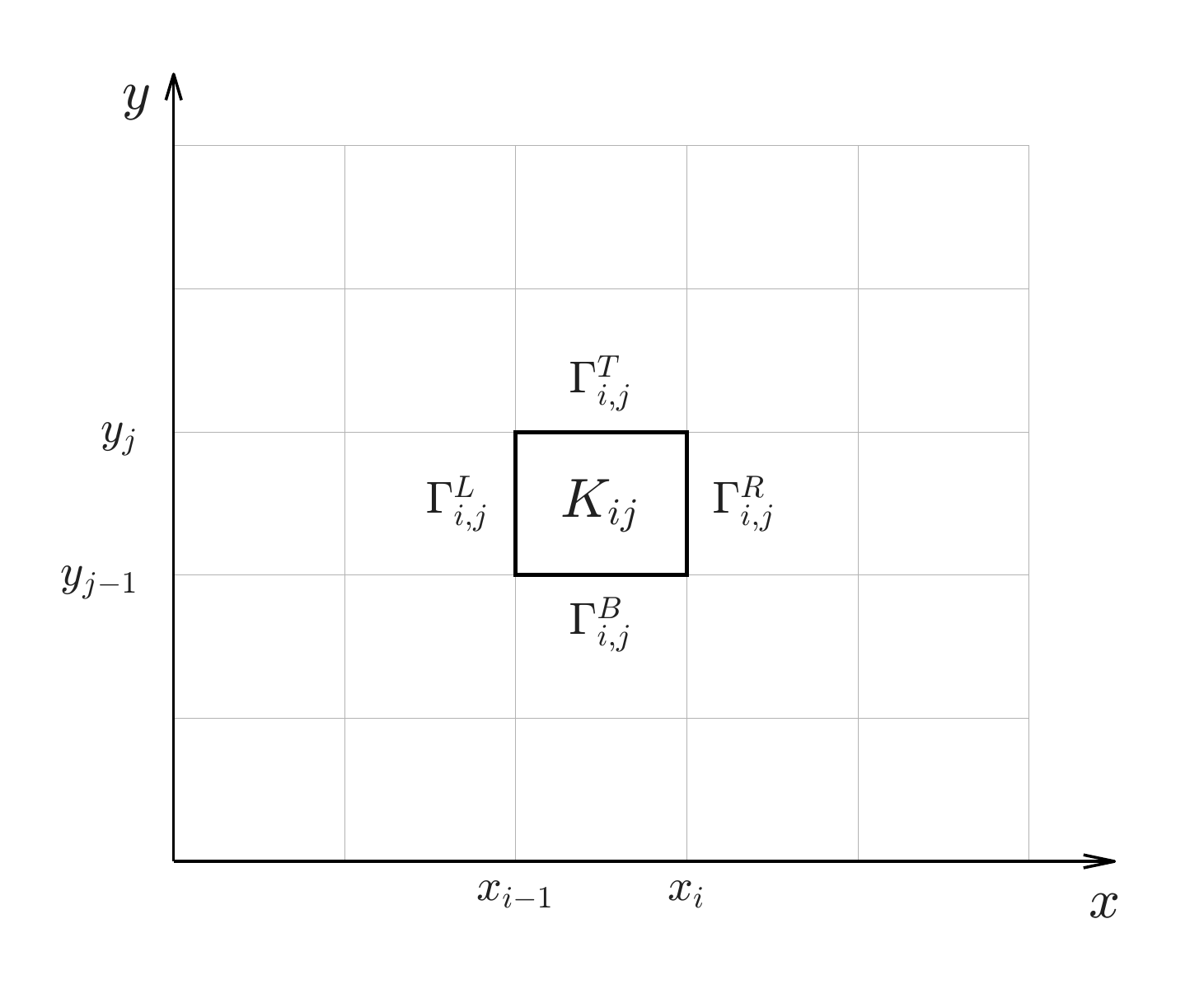}
\caption{Cartesian grid structure showing the element $K_{ij}$ with its boundary faces.}
\label{fig:cartesian_grid}
\end{figure}
The element boundary space $\partial\mathcal{T}_h$ is defined as the product of all element boundaries, which corresponds to the direct sum of the boundary sets:
\[
\partial\mathcal{T}_{h} := \bigoplus_{K \in \mathcal{T}_h} \partial K.
\]
For \(U,V\in L^2(\mathcal T_h)\), denote
\begin{equation}\label{innerprod}
    (U,V)_{K_{ij}} := \int_{K_{ij}} UV\,dxdy,
  \qquad
  (U,V)_{\mathcal{T}_h} := \sum_{K_{ij}\in\mathcal{T}_h}(U,V)_{K_{ij}}, \qquad \|V\|_{\mathcal{T}_h} := (V,V)_{\mathcal{T}_h}^{1/2}.
\end{equation}
We use the notation
\begin{equation}
\begin{aligned}\label{innerprod2}
\langle \phi, \psi  \rangle_{\partial\mathcal T_h} := \sum_{K_{ij} \in \mathcal T_h} \int_{\partial K_{ij}} \phi|_{K_{ij}} \psi|_{K_{ij}} \,ds 
=:\langle \phi,\psi\rangle_{L}+\langle \phi,\psi\rangle_{R}+\langle \phi,\psi\rangle_{B}+\langle \phi,\psi\rangle_{T},
\end{aligned}
\end{equation}
where
\begin{equation}
\begin{aligned}\label{innerprod4}
\langle \phi,\psi\rangle_{L}&= \sum_{K_{ij} \in \mathcal T_h} \Bigl[\int_{\Gamma_{ij}^L} \phi|_{K_{ij}} \psi|_{K_{ij}}\, ds, \quad \langle \phi,\psi\rangle_{R}= \sum_{K_{ij} \in \mathcal T_h} \Bigl[\int_{\Gamma_{ij}^R} \phi|_{K_{ij}} \psi|_{K_{ij}}\, ds, \\ \langle \phi,\psi\rangle_{T}&= \sum_{K_{ij} \in \mathcal T_h} \Bigl[\int_{\Gamma_{ij}^T} \phi|_{K_{ij}} \psi|_{K_{ij}}\,ds, \quad\langle\phi,\psi\rangle_{B}= \sum_{K_{ij} \in \mathcal T_h} \Bigl[\int_{\Gamma_{ij}^B} \phi|_{K_{ij}} \psi|_{K_{ij}} \,ds, 
\end{aligned}
\end{equation}
and the norms $\| \psi \|_{W} :=\langle \psi, \psi \rangle_W^{1/2}$, where $W = L,R,B,T$, additionally, we write
\begin{equation}
\begin{aligned}\label{innerprod3}
\| \psi \|^2_{\mathcal V} :=\|\psi\|^2_{L}+\|\psi\|^2_{R},\quad \text{ and  }\quad \| \psi \|^2_{\mathcal H} :=\|\psi\|^2_{B}+\|\psi\|^2_{T}.
\end{aligned}
\end{equation}

For an integer \(k\geq0\), let
\(
  P_k(S)
\)
denote the polynomials of total degree at most \(k\) on a set \(S\).
The tensor–product cell space is
\[
  V_h^k := \bigl\{v\in L^2(\Omega)\;:\;
                   v|_{K_{ij}}\in Q_k(K_{ij})\ \forall K_{ij}\in\mathcal{T}_h
                 \bigr\},  \quad \text{ where } \quad Q_k(K_{ij}) := P_k(I_i)\otimes P_k(J_j).
\]
On the skeleton, we introduce the trace space
\begin{align*}
  M_h^k := &\bigl\{\mu\in L^2(\mathcal{F}_h) \mid
                 \mu|_{F}\in P_k(F),\quad \forall F\in\mathcal{F}_h
           \bigr\},\\
   \mathcal V_h^k(g) :=& \bigl\{\gamma \in M_h^k \mid \gamma|_{\Gamma_L}=\gamma|_{\Gamma_R}=g\},\\
    \widetilde{M}_h^k(\mathcal S) :=& \bigl\{ \mu \in L^2(\mathcal{S}) \mid 
\mu|_F \in P_k(F), \quad \forall F \in \mathcal S \bigr\}, \qquad \mathcal S \subset  \mathcal F_h.
\end{align*}

\section{HDG Formulation and Stability Analysis for the CH-KP Equation}
\label{sec:formulation}

This section develops the HDG formulation and proves energy stability for the CH-KP equation \eqref{eq:ibvp} on a bounded domain. Building on the evolution form \eqref{eq:evo}, we consider the Helmholtz resolvent \((1-\partial_x^{2})^{-1}\) which provides coercivity and ensures well-posedness under homogeneous Dirichlet/Neumann boundary conditions with initial data \(u_0(x,y)=u(x,y,0)\) (see \cite{GuiLiuLuoYin2021, Pelinovsky2024}). 


We introduce auxiliary variables to decompose nonlocal operators and higher derivatives
\begin{subequations}\label{eq:local-system}
    \begin{equation}\label{eq:local-system-fixed}
        \begin{aligned}
            q &:= u_x, \qquad p := (uq)_x, \qquad s := u_y, \\
            v_x &:= s, \qquad r := z_x + f(u)_x - p_x + \tfrac{1}{2}(q^2)_x + v_y, \qquad z := r_x 
        \end{aligned}
    \end{equation}
    which implies \(v_y = \partial_x^{-1}u_{yy}\) and \(r = (1-\partial_x^{2})^{-1}[f(u)_x - p_x + \tfrac{1}{2}(q^2)_x + v_y]\). The CH-KP equation \eqref{eq:ibvp} is then equivalent to
    \begin{equation}\label{eq:local-system-fixedb}
        u_t + r = 0,
    \end{equation}
\end{subequations}
along with the initial and boundary conditions in \eqref{eq:ibvp}.
This decomposition isolates each spatial derivative for the local elimination of auxiliary variables and preserves $L^2$-stability via the Helmholtz structure of $r$ \cite{Pelinovsky2024}.

\subsubsection*{Local problems} Let us consider an element \(K_{ij}\). For the given polynomial traces \(\widehat{u}_{\Gamma_{i,j}^L}\), \(\widehat{u}_{\Gamma_{i,j}^R}\), \(\widehat{u}_{\Gamma_{i,j}^B}\), \(\widehat{q}_{\Gamma_{i,j}^L}\), \(\widehat{q}_{\Gamma_{i,j}^R}\), \(\widehat{v}_{\Gamma_{i,j}^R}\), and \(\widehat{v}_{\Gamma_{i,j}^T}\), find \((U,Q,P,S,V,Z,R) \in [Q_k(K_{ij})]^7\) that satisfies the following local system 
\begin{align*}
Q &= U_x, \qquad P = (UQ)_x, \qquad
S = U_y, \qquad V_x = S, \\
Z &= R_x, \qquad R - Z_x = \partial_x f(U) -P_x + \tfrac{1}{2}(Q^2)_x + V_y, \qquad
U_t + R = 0,
\end{align*}
with the initial condition \(U|_{t=0} = u_0|_{K_{ij}}\) and boundary conditions:
\begin{align*}
U|_F &= \widehat{u}_F,\qquad  F \in \{\Gamma_{i,j}^L, \Gamma_{i,j}^R, \Gamma_{i,j}^B\}, \\ Q|_{\Gamma_{i,j}^L} &= \widehat{q}_{\Gamma_{i,j}^L},\quad Q|_{\Gamma_{i,j}^R} = \widehat{q}_{\Gamma_{i,j}^R}, \qquad V|_{\Gamma_{i,j}^R} = \widehat{v}_{\Gamma_{i,j}^R},\quad V|_{\Gamma_{i,j}^T} = \widehat{v}_{\Gamma_{i,j}^T}.
\end{align*}
\subsubsection*{Coupling conditions} Assuming the local problem is uniquely solvable for the given sufficiently smooth initial and boundary data, the local solutions \((U,Q,P,S,V,Z,R) \in [Q_k(K_{ij})]^7\) form a global solution to \eqref{eq:local-system-fixed}-\eqref{eq:local-system-fixedb} if and only if the quantities $U,Q,UQ,V,G$ are continuous over vertical and $U, V$ are continuous over horizontal faces, where \(\mathcal{G}:=Z+f(U)-P\), and the boundary conditions in \eqref{eq:ibvp} are satisfied. 
Note that, if $U$ and $Q$ are continuous along the interfaces then $R=-U_t$ and $Q^2$ will be continuous for each fixed time $t$.

The HDG method therefore solves the local problems element-wise and enforces these transmission conditions via global coupling to determine the trace unknowns.
This approach preserves the tensor-product structure of the discretization while restricting global coupling to the skeleton. The HDG scheme is then constructed as a discrete analogue of this continuous characterization.

\subsection{An HDG scheme} For every element \(K_{ij}\in \mathcal T_h\),  
assume that the polynomial traces  
\[
\widehat{u}_{h,\Gamma_{i,j}^L},\quad \widehat{u}_{h,\Gamma_{i,j}^R}, \quad\widehat{u}_{h,\Gamma_{i,j}^B}, \quad\widehat{q}_{h,\Gamma_{i,j}^L}, \quad\widehat{q}_{h,\Gamma_{i,j}^R},\quad  \widehat{v}_{h,\Gamma_{i,j}^R},\quad \widehat{v}_{h,\Gamma_{i,j}^T}
\]
are given.  
Note that these families constitute the only unknowns that couple different cells; once they are fixed, each \(K_{ij}\) can be treated independently. Given the traces above and the restriction of the initial datum \(u_h(x,y,0) = \mathbb Pu_0\) to \(K_{ij}\), where $\mathbb P$ is the usual $L^2$-projection onto the finite element space $V_h^k$,
determine
\(
  (u_h,q_h,p_h,s_h,v_h,z_h,r_h)\;\in\;\bigl[ Q_k(K_{ij}) \bigr]^7
\)
such that, for every test function
\((\phi_u,\phi_q,\phi_p,\phi_s,\phi_v,\phi_z,\phi_r)\in\bigl[ Q_k(K_{ij}) \bigr]^7\), the following hold
\begin{subequations}\label{eq:hdg-local}
\begin{align}
  &(q_h,\phi_q)_{K_{ij}}
  +(u_h,\partial_x\phi_q)_{K_{ij}}
  -\langle \widehat u_hn_x,\phi_q\rangle_{\partial K_{ij}} = 0,
  \label{eq:hdg-local:a}\\[2pt]
  &(p_h,\phi_p)_{K_{ij}}
  +(u_h q_h,\partial_x\phi_p)_{K_{ij}}
  -\langle \widehat{u_h q_h}n_x,\phi_p\rangle_{\partial K_{ij}} = 0,
  \label{eq:hdg-local:b}\\[2pt]
  &(s_h,\phi_s)_{K_{ij}}
  +(u_h,\partial_y\phi_s)_{K_{ij}}
  -\langle \widehat u_hn_y,\phi_s\rangle_{\partial K_{ij}} = 0,
  \label{eq:hdg-local:c}\\[2pt]
  &(s_h,\phi_v)_{K_{ij}}
  +(v_h,\partial_x\phi_v)_{K_{ij}}
  -\langle \widehat v_hn_x,\phi_v\rangle_{\partial K_{ij}} = 0,
  \label{eq:hdg-local:d}\\[2pt]
  &(z_h,\phi_z)_{K_{ij}}
  +(r_h,\partial_x\phi_z)_{K_{ij}}
  -\langle \widehat r_hn_x,\phi_z\rangle_{\partial K_{ij}} = 0,
  \label{eq:hdg-local:f}\\[2pt]
  &(r_h,\phi_r)_{K_{ij}}
  + \bigl(z_h-p_h+ f(u_h),\partial_x\phi_r\bigr)_{K_{ij}} +\bigl(\tfrac12 q_h^{2},\partial_x\phi_r\bigr)_{K_{ij}}
      -\bigl\langle \widehat z_h
                -\widehat p_h+\widehat{f(u_h)},n_x
               \phi_r\bigr\rangle_{\partial K_{ij}} \nonumber\\
  &\quad-\bigl\langle\tfrac12\widehat{q_h^{2}},n_x
               \phi_r\bigr\rangle_{\partial K_{ij}}
  +(v_h,\partial_y\phi_r)_{K_{ij}} -\langle \widehat v_hn_y,\phi_r\rangle_{\partial K_{ij}}= 0,
  \label{eq:hdg-local:g}\\[2pt]
  &\bigl( (u_h)_t ,\phi_u\bigr)_{K_{ij}}
  + (r_h,\phi_u)_{K_{ij}} = 0.
  \label{eq:hdg-local:h}
\end{align}
\end{subequations}
Here we have written, for convenience, 
\(
  \widehat \omega_h = \widehat \omega_{h,F}\text{ for every }F\in\partial K_{ij}\). The remaining traces are defined below, where $\tau_{\bullet}$ denote the stabilization parameters defined on the mesh skeleton $\partial \mathcal{T}_h$. 
On the vertical faces \(\Gamma_{i,j}^L\) and  \(\Gamma_{i,j}^R\), we define
\begin{equation}\label{Vertical_trace}
\begin{aligned}
\widehat z_h- \widehat p_h &=
\begin{cases}
   z_h -  p_h + \tau_{zpu}^+(\widehat u_h-u_h)n_x  , & \text{ on } \Gamma_{i,j}^L,\\[3pt]
    z_h -p_h +\tau_{zpu}^-
(\widehat u_h-u_h)n_x +\tau_{zpv}^-
(\widehat v_h-v_h)n_x, & \text{ on } \Gamma_{i,j}^R,
   \end{cases}\\[4pt]
   \widehat{f(u_h)} & = f(u_h)-\tau_{f}(\widehat u_h-u_h)n_x,\\[4pt]
\widehat r_h &=-(\widehat u_{h})_t,\\[4pt]
        \widehat v_h &=  v_h , \qquad  \text{ on } \Gamma_{i,j}^L,
        \\[4pt]
 \frac{1}{2}\widehat{q_h^2}&=\frac{1}{2}\widehat q_h ^2,\\[4pt]
\widehat{u_hq_h}&=  u_h\left(\frac{\widehat q_h+ q_h}{2}\right) +\tau_{uqq}(\widehat q_h-q_h)n_x.
\end{aligned}
\end{equation}
On the horizontal faces \(\Gamma_{i,j}^B\) and \(\Gamma_{i,j}^T\), we define
\begin{equation}\label{Horizontal_trace}
\begin{split}
\widehat v_h &=v_h, \qquad \text{ on } \Gamma_{i,j}^B,\\
\widehat u_h &=u_h, \qquad \text{ on } \Gamma_{i,j}^T.
\end{split}
\end{equation}
The functions 
\(
\tau_{\boldsymbol{\cdot}} \text{ and } \;\tau_{f}:=\tau_{f}(\widehat u_h,u_h)
\)
are defined on the mesh skeleton \(\partial\mathcal T_h\) and are collectively referred to as the stabilization function. They have to be chosen in such a way that the discrete problem \eqref{eq:hdg-local} is stable and uniquely solvable.
In particular, because \(f(u)\) is nonlinear, the last component
\(\tau_{f}(\widehat u_h,u_h):\partial\mathcal T_h\to\mathbb R\) is allowed to be nonlinear in both of its arguments.

\subsubsection*{Transmission condition:} It remains to determine families of globally coupled unknowns
\begin{equation*}
\widehat u_h^V\in \mathcal{V}_h^k(0),\quad \widehat u_h^{B}\in \widetilde M_h^k( \mathcal H\setminus \Gamma_T),\quad \widehat q_h^V\in \mathcal{V}_h^k(0),\quad 
\widehat v_h^{R}\in\widetilde M_h^k( \mathcal V\setminus \Gamma_L),\quad 
\widehat v_h^{T}\in\widetilde M_h^k( \mathcal H\setminus \Gamma_B),
\end{equation*}
where the superscripts \(V, B, R,T\) indicate the associated spatial direction and orientation.
We let numerical traces $\widehat p_h, \widehat z_h \in \widetilde M_h^k(\partial\mathcal{T}_h)$. The globally unknown numerical traces can be found by enforcing 
the weak transmission conditions as follows:
\begin{equation}\label{eq:weak_TC}
\begin{aligned}
&\langle \widehat{u_hq_h}n_x,  \gamma\rangle_{\partial\mathcal T_h} = 0, \qquad \langle \widehat u_h n_y, \theta \rangle_{\partial\mathcal T_h} = \langle u_D n_y, \theta \rangle_{\Gamma_B},\qquad \langle \widehat v_h n_x, 
     \lambda \rangle_{\partial\mathcal T_h} = \langle v_Rn_x, \lambda \rangle_{\Gamma_R},\\&   
\langle  \widehat z_h+\widehat{f(u_h)}
             -\widehat p_h,  \nu n_x\rangle_{\partial\mathcal T_h} = 0, \qquad \langle \widehat v_hn_y, 
     \zeta  \rangle_{\partial\mathcal T_h} = \langle v_Tn_y, \zeta \rangle_{\Gamma_T},
\end{aligned}
\end{equation}
for all test traces \(\gamma, \nu \in \mathcal V^k_h(0)\), \(\theta\in\widetilde M_h^k( \mathcal H\setminus \Gamma_B)\), \(\lambda\in\widetilde M_h^k( \mathcal V\setminus \Gamma_L)\), and \(\zeta \in \widetilde M_h^k(\mathcal H\setminus \Gamma_T)\). No additional constraints are imposed on
\(\widehat u_h^x\) and  \(\widehat q_h^x\) because they are already single–valued on vertical faces by definition, and hence $\langle \widehat u_h^V n_x, \delta \rangle_{\partial\mathcal T_h} = 0$ and $\langle \widehat q_h^Vn_x, \sigma \rangle_{\partial\mathcal T_h} = 0$, for all $\delta, \sigma \in \mathcal V^k_h(0)$. Note that $n_x=0$ on the horizontal faces and $n_y=0$ on vertical faces.
\subsubsection*{Initial and boundary conditions:} On boundary faces, the numerical traces are fixed by the prescribed physical data and by an additional anchoring condition for the
primitive operator \(\partial_x^{-1}\); explicitly
\begin{equation}\label{boundary_data}
\begin{aligned}
  &\widehat u_{h}=u_D \quad \text{on }~ \partial\Omega/{\Gamma_T}; \qquad \widehat q_h=q_L \quad \text{ on }~ 
    \Gamma_L; \qquad \widehat q_h=q_R \quad \text{ on }~ 
    \Gamma_R; \qquad \widehat v_h = v_R \quad \text{ on }~ \Gamma_R;\\
    &\widehat v_h = v_T \qquad \text{ on } \Gamma_T; \qquad u_h(x,y,0)=\mathbb P  u_0(x,y) \quad \text{ for all }~(x,y)\in\Omega.
\end{aligned}
\end{equation}

Assembling the element-wise relations \eqref{eq:hdg-local} over the Cartesian mesh and employing the inner-product notation introduced in Section~\ref{sec:notation}, we obtain the following HDG formulation. 
The HDG approximations  $(u_h,q_h,p_h,s_h,v_h,w_h,z_h,r_h,\widehat u_h^V, \widehat u_h^{B}, \widehat q_h^V, \widehat v_h^{R},\widehat v_h^{T} ) \in \bigl[V_h^{k}\bigr]^7 \times \mathcal{V}_h^k(0) \times \widetilde M_h^k( \mathcal H\setminus \Gamma_T)\times \mathcal{V}_h^k(0)\times\widetilde M_h^k( \mathcal V\setminus \Gamma_L) 
\times\widetilde M_h^k( \mathcal H\setminus \Gamma_B)$ of \eqref{eq:local-system-fixed}-\eqref{eq:local-system-fixedb} satisfy
\begin{subequations}\label{eq:hdg-global}
\begin{align}
  (q_h,\phi_q)_{\mathcal T_h}
  + (u_h,\partial_x\phi_q)_{\mathcal T_h}
  -\langle \widehat u_hn_x,\phi_q\rangle_{\partial \mathcal T_h} &= 0, \label{Hdg_qh}
  \\[2pt]
  (p_h,\phi_p)_{\mathcal T_h}
  + (u_h q_h,\partial_x\phi_p)_{\mathcal T_h}
  -\langle \widehat{u_h q_h}n_x,\phi_p\rangle_{\partial \mathcal T_h} &= 0,
  \\[2pt]
  (s_h,\phi_s)_{\mathcal T_h}
  + (u_h,\partial_y\phi_s)_{\mathcal T_h}
  -\langle \widehat u_hn_y,\phi_s\rangle_{\partial \mathcal T_h} &= 0,
  \\[2pt]
  (s_h,\phi_v)_{\mathcal T_h}
  + (v_h,\partial_x\phi_v)_{\mathcal T_h}
  -\langle \widehat v_hn_x,\phi_v\rangle_{\partial \mathcal T_h} &= 0,
  \\[2pt]
  (z_h,\phi_z)_{\mathcal T_h}
  + (r_h,\partial_x\phi_z)_{\mathcal T_h}
  -\langle \widehat r_hn_x,\phi_z\rangle_{\partial \mathcal T_h} &= 0,
  \\[2pt]
  (r_h,\phi_r)_{\mathcal T_h}
  +\bigl(z_h+f(u_h)-p_h,\partial_x\phi_r\bigr)_{ \mathcal T_h} +\bigl(\tfrac12 q_h^{2},\partial_x\phi_r\bigr)_{ \mathcal T_h}  +(v_h,\partial_y\phi_r)_{\mathcal T_h}\nonumber\\
  \qquad
  -\langle \widehat z_h+\widehat{f(u_h)}
             -\widehat p_h,n_x\phi_r
             \rangle_{\partial \mathcal T_h} -\langle\tfrac12\widehat{q_h^{2}},n_x\phi_r
             \rangle_{\partial \mathcal T_h}
 -\langle \widehat v_hn_y,\phi_r\rangle_{\partial\mathcal T_h} &= 0,
  \\[2pt]
  \bigl((u_h)_t,\phi_u\bigr)_{\mathcal T_h}
  + (r_h,\phi_u)_{\mathcal T_h} &= 0,
\end{align}
\end{subequations}
for every  
\( (\phi_u,\phi_q,\phi_p,\phi_s,\phi_v,\phi_z,\phi_r)\in\bigl[V_h^{k}\bigr]^7.
\)
 The auxiliary quantities
\(\widehat p_h,\widehat r_h,\widehat z_h,\widehat{u_h q_h}, \widehat {f(u_h)}\), and $\frac{1}{2}\widehat{q_h^2}$
appearing in \eqref{eq:hdg-global} are prescribed face-wise
by \eqref{Vertical_trace} and \eqref{Horizontal_trace} in terms of the numerical traces \(\widehat u_h^x, \widehat u_h^{y+},\widehat q_h^x, \widehat v_h^{x^-},\widehat v_h^{y^-}\). These globally coupled trace unknowns can be determined by the transmission condition \eqref{eq:weak_TC}.

Since \(f(u)=2\kappa u+\frac32u^{2}\) is the nonlinearity present in the CH–KP
flux,  we define it on every (oriented) face \(F\subset\partial\mathcal T_h\),
\begin{equation}\label{eq:tildetau}
  \tilde\tau(\widehat u_h,u_h)
  :=\frac{1}{(\widehat u_h-u_h)^{2}}
     \int^{u_h}_{\widehat u_h}
       \bigl(f(s)-f( u_h)\bigr)n_x\mathrm ds.
\end{equation}
Therefore, we have
\[
  \bigl\lvert\tilde\tau(\widehat u_h,u_h)\bigr\rvert
  \;\le\;
  \tfrac12\sup_{s\in J(\widehat u_h,u_h)}
            \bigl\lvert f'(s)\bigr\rvert ,
\]
where $J(\widehat u_h,u_h) = [\min(u_h,\widehat u_h),
                     \max(u_h,\widehat u_h)]$. In order to ensure stability, we need the point-wise inequality
\(
  \tau_f-\tilde\tau\ge0
\)
(the first component of our collective stabilization function), for which it
is sufficient to choose, e.g.
\(
  \tau_f=\sup_{s\in J(\widehat u_h,u_h)}\tfrac12\bigl\lvert f'( s)\bigr\rvert+\varepsilon
\) for some $\varepsilon>0$. Additionally, we assume that the stabilization parameters satisfy the following global assumption.
\begin{assumption}[A global assumption on the stabilization parameters]\label{globalass}
We assume that the stabilization parameters in the scheme \eqref{Vertical_trace}-\eqref{Horizontal_trace} satisfy
    \begin{align}
  \tau_f-\tilde\tau \geq C_l,\quad -\tau_{zpu}^+\geq 0, \quad -\tau_{zpu}^-\geq 0, \quad (\tau_{zpv}^-)^2\leq -2\tau_{zpu}^-, \quad -\tau_{uqq} \geq C\delta, \label{ass3}
  \end{align}
  for some positive constants $C$ and $C_l$ (we choose $C_l$ large enough and $\delta >0$ small enough for error analysis).
\end{assumption}

Please note that for the stability analysis, we do not need the strict global Assumption \ref{globalass} completely; we can relax this and take $\tau_f = \sup_{s\in J(\widehat u_h,u_h)}\tfrac12\bigl\lvert f'( s)\bigr\rvert+\varepsilon$, as mentioned earlier, and $-\tau_{uqq}\geq 0$, along with other assumptions.
\subsection{Stability analysis}
Define the discrete energy functional
\[
  \mathcal E_h(t):=\frac12\Bigl(
        \|u_h(t)\|_{\mathcal T_h}^2+\|q_h(t)\|_{\mathcal T_h}^2\Bigr), \qquad\qquad \forall t\geq 0.
\]
\begin{lemma}[Stability lemma]\label{lem:stability}
Assume that the boundary condition in \eqref{boundary_data} is set to be homogeneous, i.e., $u_D=q_L=q_R=v_R=v_T =0$ and that the stabilization parameters satisfy Assumption \ref{globalass}.
Then the HDG approximation
\(
  (u_h,q_h,p_h,s_h,v_h,z_h,r_h),
\)
obtained from \eqref{eq:hdg-local}-\eqref{boundary_data} satisfies the energy inequality
\begin{equation}\label{Energy_stab}
  \frac{\mathrm d}{\mathrm dt}\mathcal E_h(t)\le 0,
  \qquad\text{for all }t>0 .
\end{equation}
Consequently, $\mathcal E_h(t)\le \mathcal E_h(0)$ and the scheme is
$L^2$–stable.
\end{lemma}

\begin{proof}
Differentiating \eqref{Hdg_qh} with respect to time  and setting $\phi_q = q_h$ yields
\begin{equation}\label{Hdg_qht}
  \bigl((q_h)_t,q_h\bigr)_{\mathcal T_h}
  +\bigl((u_h)_t,\partial_x q_h\bigr)_{\mathcal T_h}
  -\langle (\widehat u_h)_t n_x,q_h\rangle_{\partial\mathcal T_h}=0 ,
\end{equation}
where we employ the inner-product notation from \eqref{innerprod}-\eqref{innerprod2}.
Next, choose the test functions
\[
  (\phi_u,\phi_q,\phi_p,\phi_s,\phi_v,\phi_z,\phi_r)
  =(u_h-\partial_xq_h,p_h-z_h,-q_h,-v_h,v_h,q_h,-u_h),
\]
in the system \eqref{eq:hdg-global}, and adding the equations in the resulting system and \eqref{Hdg_qht} yields
    \begin{equation}
        \begin{split}
   &         \big((u_{h})_t,u_h\big)_{\mathcal T_h}+\big((q_{h})_t,q_h\big)_{\mathcal T_h}
      -\langle(\widehat u_{h})_tn_x,q_h\rangle_{\partial\mathcal T_h} 
  +(u_h,\partial_xp_h)_{\mathcal T_h}-\langle \widehat u_hn_x,p_h\rangle_{\partial \mathcal T_h}\\
  &\quad- (u_h,\partial_xz_h)_{\mathcal T_h}+\langle \widehat u_hn_x,z_h\rangle_{\partial \mathcal T_h}
  - (u_h q_h,\partial_xq_h)_{\mathcal T_h}
  +\langle \widehat{u_h q_h}n_x,q_h\rangle_{\partial \mathcal T_h} 
  - (u_h,\partial_yv_h)_{\mathcal T_h}\\
  &\quad+\langle \widehat u_hn_y,v_h\rangle_{\partial \mathcal T_h}
  + (v_h,\partial_x v_h)_{\mathcal T_h}
  -\langle \widehat v_hn_x,v_h\rangle_{\partial \mathcal T_h}
  - (v_h,\partial_yu_h)_{\mathcal T_h}
  +\langle \widehat v_hn_y,u_h\rangle_{\partial \mathcal T_h}
  \\
  &\quad
  -\langle \widehat r_hn_x,q_h\rangle_{\partial \mathcal T_h} 
  -\bigl(z_h+f(u_h)-p_h,\partial_xu_h\bigr)_{ \mathcal T_h}\nonumber +\bigl(\tfrac12 q_h^{2},\partial_xu_h\bigr)_{ \mathcal T_h}\\&\quad
  +\langle \widehat z_h+\widehat{f(u_h)}
             -\widehat p_h,n_xu_h
             \rangle_{\partial \mathcal T_h} +\tfrac12\langle\widehat{q_h^{2}}n_x,u_h
             \rangle_{\partial \mathcal T_h} = 0.
        \end{split}
    \end{equation}
  Applying multiple integrations by parts yields
    \begin{equation}
        \begin{split}
   &         \mathcal E_h(t)
      -\langle(\widehat u_{h})_tn_x,q_h\rangle_{\partial\mathcal T_h} 
  +\langle u_h,n_xp_h\rangle_{\partial \mathcal T_h}-\langle \widehat u_hn_x,p_h\rangle_{\partial \mathcal T_h}- \langle u_h,n_xz_h\rangle_{\partial \mathcal T_h}\\
  &\quad+\langle \widehat u_hn_x,z_h\rangle_{\partial \mathcal T_h}
  - \Big\langle u_h, n_x\frac{1}{2}(q_h)^2\Big\rangle_{\partial\mathcal T_h}
  +\langle \widehat{u_h q_h},n_xq_h\rangle_{\partial \mathcal T_h} 
  - \langle u_h,n_yv_h\rangle_{\partial \mathcal T_h}\\
  &\quad+\langle \widehat u_hn_y,v_h\rangle_{\partial \mathcal T_h}
  + \frac{1}{2}\langle v_h, n_x v_h\rangle_{\partial \mathcal T_h}
  -\langle \widehat v_h,n_xv_h\rangle_{\partial \mathcal T_h}
  +\langle \widehat v_hn_y,u_h\rangle_{\partial \mathcal T_h}
  \\
  &\quad
  -\langle \widehat r_hn_x,q_h\rangle_{\partial \mathcal T_h} 
  -\bigl(f(u_h),\partial_xu_h\bigr)_{ \mathcal T_h}
  +\langle \widehat z_h+\widehat{f(u_h)}
             -\widehat p_h,n_xu_h
             \rangle_{\partial \mathcal T_h}  +\langle\tfrac12\widehat{q_h^{2}},n_xu_h
             \rangle_{\partial \mathcal T_h}= 0.
        \end{split}
    \end{equation}
  The transmission condition \eqref{eq:weak_TC}, and the boundary conditions \eqref{boundary_data} lead to
    \begin{equation}\label{energy}
        \begin{split}
   &         \mathcal E_h(t) \underbrace{-\bigl(f(u_h),\partial_xu_h\bigr)_{\mathcal T_h}-\langle\widehat{f(u_h)},(\widehat u_h -u_h)n_x\rangle_{\partial \mathcal T_h}}_{\mathcal I_f}
  \\
  &\quad  \underbrace{- \langle\widehat u_h-u_h, ((\widehat z_h- \widehat p_h)-(z_h-p_h))n_x\rangle_{\partial \mathcal T_h}+\frac{1}{2}\langle(\widehat v_h -v_h)^2,n_x\rangle_{\partial \mathcal T_h} }_{\mathcal I_{u_x}}+\frac{1}{2}\widehat v_h ^2|_{\Gamma_L}
  \\
  &\quad +\underbrace{\left\langle \frac{1}{2}\widehat{q_h^2}-\frac{1}{2}q_h^2,u_hn_x\right\rangle_{\partial \mathcal T_h}  + \langle\widehat {u_hq_h},q_hn_x\rangle_{\partial \mathcal T_h}}_{\mathcal I_{q}} \underbrace{-\langle\widehat u_h-u_h, (\widehat v_h-v_h)n_y\rangle_{\partial \mathcal T_h}}_{\mathcal I_{u_y}}
 = 0.
     \end{split}
    \end{equation}
   Let $G$ be an antiderivative of $f$, i.e. $G'(s)=f(s)$.  Since
    \begin{equation*}
    -\bigl(f(u_h),\partial_xu_h\bigr)_{\mathcal T_h}\nonumber = -\bigl(
        \partial_xG(u_h),1\bigr)_{\mathcal T_h}\nonumber = -\langle G(u_h),n_x\rangle_{\partial \mathcal T_h}=-\left\langle\int_{\widehat u_h }^{u_h}f(s)ds,n_x
        \right\rangle_{\partial \mathcal T_h},
    \end{equation*}
   we obtain
    \begin{align}\label{f_estimate}
        \mathcal{I}_f&=-\left\langle\int_{\widehat u_h }^{u_h}(f(s)-f(u_h ))ds,n_x\right\rangle_{\partial \mathcal T_h}- \left\langle\widehat{f(u_h)}-f( u_h ),(\widehat u_h -u_h)n_x\right\rangle_{\partial \mathcal T_h}\nonumber\\
        &=\langle \tau_f-\tilde \tau, (n_x)^2 (\widehat u_h -u_h)^2\rangle_{\partial\mathcal T_h},
    \end{align}
    where we employed \eqref{eq:tildetau}. Using the transmission condition \eqref{eq:weak_TC} and numerical traces \eqref{Vertical_trace}, we evaluate $\mathcal{I}_q$ as follows
\begin{align*}
  \mathcal I_q
  &=\Bigl\langle \tfrac12(\widehat q_h)^2-\tfrac12 q_h^2,
          u_h n_x\Bigr\rangle_{\partial \mathcal T_h}
    -\langle \widehat{u_h q_h},(\widehat q_h-q_h)n_x\rangle_{\partial \mathcal T_h}\\
  &=\left\langle u_h\left( \frac{\widehat q_h+q_h}{2}\right)
         -\widehat{u_h q_h},(\widehat q_h-q_h)n_x\right\rangle_{\partial \mathcal T_h}=-\langle \tau_{uqq},(n_x)^2(\widehat q_h-q_h)^2\rangle_{\partial\mathcal T_h}.
\end{align*}
We evaluate $\mathcal{I}_{u_x}$ by applying \eqref{Vertical_trace}, which implies
\begin{align*}
  \mathcal I_{u_x}&=-\Bigl\langle \tau_{zpu}^+,
               (\widehat u_h-u_h)^2\Bigr\rangle_{L}-\langle \tau_{zpu}^-,
            (\widehat u_h-u_h)^2\rangle_{R}-\langle \tau_{zpv}^-,
            (\widehat u_h-u_h)(\widehat v_h^{x^-}-v_h)\rangle_{R}\\& \qquad
  +\Bigl\langle \frac{1}{2},
               (\widehat v_h^{x^-}-v_h)^2\Bigr\rangle_{R},
\end{align*}
    We finally evaluate $\mathcal{I}_{u_y}$ using \eqref{Horizontal_trace}, which yields
    \begin{align*}
        \mathcal I_{u_y} = -\langle\widehat u_h-u_h, (\widehat v_h-v_h)n_y\rangle_{B}- \langle\widehat u_h-u_h, (\widehat v_h-v_h)n_y\rangle_{T} = 0.
    \end{align*}
Invoking the assumptions \eqref{ass3} yields
\begin{align}\label{energy_2}
    \mathcal{I}_f + \mathcal{I}_{u_x} + \mathcal{I}_{q} + \mathcal I_{u_y} \geq 0.
\end{align}
Therefore, from \eqref{energy} and \eqref{energy_2}, it follows that
\[
  \frac{\mathrm{d}}{\mathrm{d}t} \mathcal{E}_h(t) \leq 0, \quad \text{for all } t > 0,
\]
which completes the proof.
\end{proof}


\section{Error analysis}\label{sec:error}
We assume homogeneous boundary conditions~\eqref{boundary_data} and that Assumption~\ref{globalass} holds.  
Let \((u_h, q_h, p_h, s_h, v_h, z_h, r_h)\) be the semi-discrete HDG approximation defined by 
\eqref{eq:hdg-local}--\eqref{boundary_data}.  
This section establishes an \textit{a priori} \(L^2\)-error estimate, demonstrating convergence of order \(h^{k+1/2}\).

The analysis relies on suitably defined projection operators.  
For an interval \(I:=(a,b)\) and \(\omega\in L^2(I)\), denote by \(\mathcal{P}\omega\in P_k(I)\) the standard \(L^2\)-projection, which satisfies
\[
\int_{I}(\mathcal{P}\omega-\omega)v\,\mathrm{d}x = 0,\qquad \forall v\in P_k(I).
\]
When \(\omega\) has additional regularity, \(\omega\in H^{\frac12+\epsilon}(I)\), $\epsilon>0$, we also use the one-sided projections \(\mathcal{P}^\pm\omega\in P_k(I)\) defined by
\begin{align*}
\int_{I}(\mathcal{P}^+\omega-\omega)v\,\mathrm{d}x &= 0,\;\forall v\in P_{k-1}(I),\quad \mathcal{P}^+\omega(a^+)=\omega(a),\\
\int_{I}(\mathcal{P}^-\omega-\omega)v\,\mathrm{d}x &= 0,\;\forall v\in P_{k-1}(I),\quad \mathcal{P}^-\omega(b^-)=\omega(b).
\end{align*}
These one-dimensional projections are well-defined for \(\omega\in H^{\frac12+\epsilon}(I)\); see \cite{cockburn2001superconvergence}.

On a rectangular cell \(K_{ij}\), define the tensor-product \(L^2\)-projection \(\mathbb{P}:=\mathcal{P}_x\otimes\mathcal{P}_y\) onto \(Q_k(K_{ij})\), where \(\mathcal{P}_x\) and \(\mathcal{P}_y\) are the standard $L^2$-projections in one dimension on the \(x\)- and \(y\)-directions, respectively.  
For \(\omega\in L^2(K_{ij})\) the projection error $\mu^{\omega}:=\mathbb{P}\omega-\omega$ satisfies
\begin{equation}\label{eq:proj-moment}
(\mu^{\omega},\psi_\omega)_{K_{ij}} = 0,\qquad \forall\psi_\omega\in Q_k(K_{ij}).
\end{equation}
For smoother functions, \(\omega\in H^2(K_{ij})\), we introduce the composite projections
\[
\Pi^{\pm}\omega:=\mathcal{P}_x^{\pm}\otimes\mathcal{P}_y^{\pm}\omega,
\]
well-defined in \(Q_k(K_{ij})\) \cite[Section~3.2]{cockburn2001superconvergence}. 
For  \(u,v\in H^2(K_{ij})\), we use the projection \(\Pi^-u,\Pi^+v\in Q_k(K_{ij})\).  
The projections \(\Pi^-u\) and \(\Pi^+v\) satisfy the following orthogonality and trace properties \cite{XuShu2007}:
\begin{equation}\label{proj_ortho_mod}
\begin{aligned}
(\eta^\omega, \phi_\omega)_{K_{ij}} = 0, \qquad \forall \phi_\omega\in (P_{k}(I_i)\otimes P_{k-1}(J_j)) \cup (P_{k-1}(I_i)\otimes P_{k}(J_j)), \quad \omega\in \{u,v\},
\end{aligned}
\end{equation}
and
\begin{equation}\label{eq:proj-vert-left}
\begin{aligned}
\langle \eta^v, \phi_v \rangle_{\Gamma_{i,j}^L} &= 0, \qquad
\langle \eta^v, \phi_v \rangle_{\Gamma_{i,j}^B} = 0, &\quad \forall \phi_v\in Q_k(K_{ij}),\\
\langle \eta^u, \phi_u \rangle_{\Gamma_{i,j}^R} &= 0,\qquad 
\langle \eta^u, \phi_u \rangle_{\Gamma_{i,j}^T} = 0, &\quad \forall \phi_u\in Q_k(K_{ij}),
\end{aligned}
\end{equation}
where 
$$ \eta^v:= \Pi^+ v-v, \qquad \eta^u: = \Pi^- u-u. $$  Throughout the analysis, we denote the projection errors for the seven variables by
\[
\eta^u,\;\mu^q,\;\mu^p,\;\mu^s,\;\eta^v,\;\mu^z,\;\mu^r,
\]
corresponding to \(u,q,p,s,v,z,r\), respectively.
\begin{lemma}[Inverse inequalities \cite{Ciarlet1978Fem,XuShu2007}]\label{lem:inverse_2d}
For every $v_h\in V_h^{k}$ $(k\ge0)$, there exists a positive constant $C$ independent
of $v_h$ and $h$, such that the following estimates hold
\begin{equation}\label{invesre}
\begin{aligned}
\|\partial_x v_h\|_{\mathcal T_h}
  &\le Ch^{-1}\|v_h\|_{\mathcal T_h},
&
\|\partial_y v_h\|_{\mathcal T_h}
  &\le Ch^{-1}\|v_h\|_{\mathcal T_h},
\\[2pt]
\|v_h\|_{\partial \mathcal T_h}
  &\le Ch^{-\frac{1}{2}}\|v_h\|_{\mathcal T_h}, & \|v_h\|_{L^\infty(\mathcal T_h)}
  &\le Ch^{-1}\|v_h\|_{\mathcal T_h}.
  \end{aligned}
\end{equation}

\end{lemma}

\begin{lemma}[Interpolation inequality \cite{cockburn2001superconvergence,Ciarlet1978Fem}]\label{lem:approx_2d}
For any $\omega\in H^{k+1}(\mathcal{T}_h)$, $k\geq 1$, there exists a constant
$C>0$, independent of $h$, such that
\begin{equation}\label{eq:proj-est}
     \|\mu^\omega\|_{\mathcal T_h}
  +
  h\|\mu^\omega\|_{L^\infty(\mathcal T_h)}
  +
  h^{\frac12}\|\mu^\omega\|_{\partial\mathcal T_h}
  \le
  Ch^{k+1}\qquad \text{ for } \omega \in \{q,p,s,z,r\}.
\end{equation}
\end{lemma}

\begin{theorem}[See \cite{cockburn2001superconvergence,Ciarlet1978Fem,XuShu2007}]\label{thm_mod_proj}
For $\omega \in H^{k+1}(\mathcal T_h)$, $k\geq 1$, there exists a constant $C > 0$ independent of $h$ such that
\begin{align}\label{mod_proj_est}
\| \eta^\omega \|_{\mathcal T_h} +h\| \eta^\omega \|_{L^\infty(\mathcal T_h)} +h^{\frac{1}{2}}\| \eta^\omega \|_{\partial\mathcal T_h} \leq C h^{k+1} |\omega|_{H^{k+1}(\mathcal T_h)}, \qquad \text{ for } \omega \in \{u,v\}.
\end{align}
\end{theorem}

Since the exact solution $(u,q,p,s,v,w,r)$ of the CH-KP system \eqref{eq:local-system-fixed}-\eqref{eq:local-system-fixedb} satisfies the HDG scheme \eqref{eq:hdg-local}-\eqref{boundary_data},  we have the following error equations:
\begin{subequations}\label{eq:error-equations}
\begin{align}
(\varepsilon_q,\phi_q)_{\mathcal T_h}
+(\varepsilon_u,\partial_x\phi_q)_{\mathcal T_h}
-\langle \widehat{\varepsilon}_u, n_x\phi_q\rangle_{\partial\mathcal T_h} &= 0,\label{hdg_eq}\\[2pt]
(\varepsilon_p,\phi_p)_{\mathcal T_h}
+\bigl(uq-u_hq_h,\partial_x\phi_p\bigr)_{\mathcal T_h}
-\langle u q-\widehat{u_hq_h},n_x\phi_p\rangle_{\partial\mathcal T_h} &= 0,\\[2pt]
(\varepsilon_s,\phi_s)_{\mathcal T_h}
+(\varepsilon_u,\partial_y\phi_s)_{\mathcal T_h}
-\langle \widehat{\varepsilon}_u, n_y\phi_s\rangle_{\partial\mathcal T_h} &= 0,\\[2pt]
(\varepsilon_s,\phi_v)_{\mathcal T_h}
+(\varepsilon_v,\partial_x\phi_v)_{\mathcal T_h}
-\langle \widehat{\varepsilon}_v, n_x\phi_v\rangle_{\partial\mathcal T_h} &= 0,\\[2pt]
(\varepsilon_z,\phi_z)_{\mathcal T_h}
+(\varepsilon_r,\partial_x\phi_z)_{\mathcal T_h}
-\langle \widehat{\varepsilon}_r, n_x\phi_z\rangle_{\partial\mathcal T_h} &= 0,\\[2pt]
(\varepsilon_r,\phi_r)_{\mathcal T_h}
+\bigl(\varepsilon_z-\varepsilon_pf(u)-f(u_h)+\tfrac12(q^{2}-q_h^{2}),
          \partial_x\phi_r\bigr)_{\mathcal T_h} \nonumber\\
+(\varepsilon_v,\partial_y\phi_r)_{\mathcal T_h}
-\Bigl\langle
  \widehat{\varepsilon}_z-\widehat{\varepsilon}_p+f(u)-\widehat{f(u_h)}+\tfrac12(q^{2}-\widehat{q_h^{2}}),
  n_x\phi_r\Bigr\rangle_{\partial\mathcal T_h}
-\langle \widehat{\varepsilon}_v, n_y\phi_r\rangle_{\partial\mathcal T_h} &= 0,\\[4pt]
((\varepsilon_u)_t,\phi_u)_{\mathcal T_h}
+(\varepsilon_r,\phi_u)_{\mathcal T_h} &= 0,
\end{align}
\end{subequations}
for every test function
\((\phi_u,\phi_q,\phi_p,\phi_s,\phi_v,\phi_z,\phi_r)\in\bigl[ Q_k(K_{ij}) \bigr]^7\). 
Here $\varepsilon_\omega := \omega - \omega_h$ for $\omega= u,q,p,v,z,s,r$ and the associated numerical trace errors \(
\widehat{\varepsilon}_\omega = P^\partial\omega-\widehat \omega_h\), where \(\omega = u,\,v,\,r,\,z,\,p,\,q
\), $\widehat \omega_h$ are defined by \eqref{Vertical_trace}-\eqref{Horizontal_trace}, and $P^\partial $ is the elementwise $L^2$-projection onto $P_k(F)$ such that, for all \(\omega\in L^2(F)\), there holds
\[
\int_{F}(P^\partial\omega-\omega)v\,\mathrm{d}x = 0,\qquad \forall v\in P_k(F) \text{ and } F \in \mathcal F_h.
\]
 Moreover, from \eqref{eq:weak_TC}, we also have the following transmission condition 
\begin{equation}\label{eq:weak_TCe}
\begin{aligned}
&\langle (uq- \widehat{u_hq_h})n_x,  \gamma \rangle_{\partial\mathcal T_h} = 0, \qquad \langle \widehat \varepsilon_u n_y, \theta  \rangle_{\partial\mathcal T_h} = 0,\qquad \langle \widehat \varepsilon_v n_x, 
     \lambda  \rangle_{\partial\mathcal T_h} = 0,\\&   
\langle  (\widehat \varepsilon_z-\widehat \varepsilon_p+ (f(u) -\widehat{f(u_h)}))n_x
             ,  \nu \rangle_{\partial\mathcal T_h} =0,\qquad \langle \widehat \varepsilon_vn_y, 
     \zeta  \rangle_{\partial\mathcal T_h} = 0, 
\end{aligned}
\end{equation}
for all test traces \(\gamma, \nu \in \mathcal V^k_h(0)\), \(\theta\in\widetilde M_h^k( \mathcal H\setminus \Gamma_B)\), \(\lambda\in\widetilde M_h^k( \mathcal V\setminus \Gamma_L)\), and \(\zeta \in \widetilde M_h^k( \mathcal H\setminus \Gamma_T)\). By the definition of $\widehat u_h$ and $\widehat q_h$, as well as the fact that the exact solution satisfies scheme \eqref{eq:hdg-local}-\eqref{boundary_data}, we have $\langle\widehat \varepsilon_u,\delta n_x\rangle_{\partial\mathcal T_h} = 0$  and $\langle\widehat \varepsilon_q,\sigma n_x\rangle_{\partial\mathcal T_h} = 0$, for all $\delta, \sigma \in \mathcal V^k_h(0)$.

Now we decompose the errors in the following way:
\begin{equation}\label{decomp_errors}
\begin{aligned}
\varepsilon_u  &= (\Pi^- u -u_h) -(\Pi^- u-u)=:\xi^u- \eta^u,\\ \varepsilon_v  &= (\Pi^+ v -v_h) -(\Pi^+ v-v)=:\xi^v- \eta^v, \\
\varepsilon_{\bar\omega}  &= (\mathbb P {\bar\omega} -{\bar\omega}_h) -(\mathbb P {\bar\omega}-{\bar\omega})=:\rho^{\bar\omega} - \mu^{\bar\omega}, \qquad {\bar\omega}\in\{q,p,s,z,r\}.
\end{aligned}
\end{equation}
We also introduce some notations on mesh skeleton $\partial \mathcal T_h$ as follows
\begin{equation}
\widehat{\xi}^\omega:= \widehat{\varepsilon}_\omega,  \qquad
 \omega \in \{u,v\} \qquad \text{ and }\qquad \widehat{\rho}^{\bar\omega}:= \widehat{\varepsilon}_{\bar\omega}, \qquad 
 \bar\omega \in \{q,p,z,r\}.
\label{eq:primary_flux_errors}
\end{equation}
 Then from the definitions \eqref{Vertical_trace}-\eqref{Horizontal_trace}, \eqref{eq:primary_flux_errors}, \eqref{eq:proj-vert-left}, and along with some simple algebraic manipulations, we have the following. On the left vertical faces ($\Gamma_{i,j}^L$, $n_x = -1$), we have
\begin{equation}
\begin{aligned}
\langle\widehat{\rho}^z-\widehat{\rho}^p,\nu_1n_x\rangle_{\Gamma_{i,j}^L} &= \langle\rho^z-\rho^p + \tau_{zpu}^{+} (\widehat{\xi}^u - \xi^u) n_x, \nu_1n_x\rangle_{\Gamma_{i,j}^L} - \langle\mu^z-\mu^p - \tau_{zpu}^{+}\eta^u n_x, \nu_1n_x\rangle_{\Gamma_{i,j}^L}, \\
\langle\widehat{\xi}^v,\nu_2n_x\rangle_{\Gamma_{i,j}^L} &= \langle\xi^v,\nu_2n_x\rangle_{\Gamma_{i,j}^L},\\
\langle\widehat{\rho^r} ,\nu_3n_x\rangle_{\Gamma_{i,j}^L}&=\langle-(\widehat \xi^u)_t,\nu_3n_x\rangle_{\Gamma_{i,j}^L},
\end{aligned}
\label{eq:vertical_left_flux}
\end{equation}
 on the right vertical face ($\Gamma_{i,j}^R$, $n_x = 1$), we have
\begin{equation}
\begin{aligned}
\langle\widehat{\rho}^z-\widehat{\rho}^p,\nu_1n_x\rangle_{\Gamma_{i,j}^R} &= \langle\rho^z-\rho^p + \tau_{zpu}^{-} (\widehat{\xi}^u - \xi^u) n_x+ \tau_{zpv}^{-} (\widehat{\xi}^v - \xi^v) n_x, \nu_1n_x\rangle_{\Gamma_{i,j}^R} \\&\qquad- \langle\mu^z-\mu^p - \tau_{zpu}^{-}\eta^u n_x- \tau_{zpv}^{-}\eta^v n_x, \nu_1n_x\rangle_{\Gamma_{i,j}^R},\\
\langle\widehat{\rho^r} ,\nu_3n_x\rangle_{\Gamma_{i,j}^R}&=\langle-(\widehat \xi^u)_t,\nu_3n_x\rangle_{\Gamma_{i,j}^R},
\end{aligned}
\label{eq:vertical_right_flux}
\end{equation}
for all $\nu_1,\nu_2,\nu_3 \in Q_k(K_{ij})$, and finally, on the horizontal faces $ \Gamma_{i,j}^B$ and $ \Gamma_{i,j}^T$, we have
\begin{equation}
\langle\widehat{\xi}^v, \gamma_1n_y\rangle_{\Gamma_{i,j}^B} := \langle\xi^v, \gamma_1 n_y\rangle_{\Gamma_{i,j}^B}\qquad  \langle\widehat{\xi}^u,\gamma_2n_y\rangle_{\Gamma_{i,j}^T} := \langle\xi^u,\gamma_2n_y\rangle_{\Gamma_{i,j}^T},
\label{eq:horizontal_flux}
\end{equation}
for all $\gamma_1,\gamma_2 \in Q_k(K_{ij})$, respectively.
Therefore, utilizing the decomposition \eqref{decomp_errors} and definition \eqref{eq:primary_flux_errors}, we rewrite the error equation in terms of $\xi$, $\eta$, $\rho$, and $\mu$ as follows:
 \begin{subequations}\label{eq:error-equationsxi}
\begin{align}
(\rho^q,\phi_q)_{\mathcal T_h}
+(\xi^u,\partial_x\phi_q)_{\mathcal T_h}
-\langle \widehat\xi^u, n_x\phi_q\rangle_{\partial\mathcal T_h} &= 0,\label{hdg_eqxi}\\[2pt]
(\rho^p,\phi_p)_{\mathcal T_h}
+\bigl(uq-u_hq_h,\partial_x\phi_p\bigr)_{\mathcal T_h}
-\langle u q-\widehat{u_hq_h},n_x\phi_p\rangle_{\partial\mathcal T_h} &= 0,\\[2pt]
(\rho^s,\phi_s)_{\mathcal T_h}
+(\xi^u,\partial_y\phi_s)_{\mathcal T_h}
-\langle \widehat\xi^u, n_y\phi_s\rangle_{\partial\mathcal T_h} &= 0,\\[2pt]
(\rho^s,\phi_v)_{\mathcal T_h}
+(\xi^v,\partial_x\phi_v)_{\mathcal T_h}
-\langle \widehat\xi^v, n_x\phi_v\rangle_{\partial\mathcal T_h} &= 0,\\[2pt]
(\rho^z,\phi_z)_{\mathcal T_h}
+(\rho^r,\partial_x\phi_z)_{\mathcal T_h}
-\langle \widehat\rho^r, n_x\phi_z\rangle_{\partial\mathcal T_h} &= 0,\\[2pt]
(\rho^r,\phi_r)_{\mathcal T_h}
+(\rho^z-\rho^p+f(u)-f(u_h)+\tfrac12(q^{2}-q_h^{2}),
          \partial_x\phi_r\bigr)_{\mathcal T_h} \nonumber
+(\xi^v,\partial_y\phi_r)_{\mathcal T_h}\\
-\bigl\langle
  \widehat\rho^z-\widehat\rho^p+f(u)-\widehat{f(u_h)},
  n_x\phi_r\bigr\rangle_{\partial\mathcal T_h}- \bigl\langle \tfrac12(q^{2}-\widehat{q_h^{2}}),
  n_x\phi_r\bigr\rangle_{\partial\mathcal T_h}
-\langle \widehat\xi^v, n_y\phi_r\rangle_{\partial\mathcal T_h} &= 0,\\[4pt]
((\xi^u)_t,\phi_u)_{\mathcal T_h}-((\eta^u)_t,\phi_u)_{\mathcal T_h}
+(\rho^r,\phi_u)_{\mathcal T_h} &= 0,
\end{align}
\end{subequations}
for every test function
\((\phi_u,\phi_q,\phi_p,\phi_s,\phi_v,\phi_z,\phi_r)\in\bigl[ Q_k(K_{ij}) \bigr]^7\), where projection orthogonality properties \eqref{eq:proj-moment}-\eqref{proj_ortho_mod} is incorporated.
 Differentiating \eqref{hdg_eq} with respect to time and taking $\phi_q = \rho^q$, we obtain
\begin{equation}\label{hdg_eqt}
    ((\rho^q)_t,\rho^q)_{\mathcal T_h}
+((\xi^u)_t,\partial_x\rho^q)_{\mathcal T_h}
-\langle (\widehat\xi^u)_t ,n_x\rho^q\rangle_{\partial\mathcal T_h} = 0,
\end{equation}
where we again used the definitions \eqref{decomp_errors} and \eqref{eq:primary_flux_errors}, and projection orthogonality \eqref{eq:proj-moment}.     
Now we write the compact form of the error equations \eqref{eq:error-equationsxi} by adding all the equations from system \eqref{eq:error-equationsxi} and \eqref{hdg_eqt}  as (we use notations $\boldsymbol{\xi} = (\xi^u,\rho^q,\rho^p,\rho^s,\xi^v,\rho^z,\rho^r)$ and $\boldsymbol{\eta}=(\eta^u,\mu^q,\mu^p,\mu^s,\eta^v,\mu^z,\mu^r)$)
\begin{align}\label{err_eqn_cpt_xi}
    \nonumber\mathcal{B}(\boldsymbol{\xi};\phi_u,\phi_q,\phi_p,\phi_s,\phi_v,\phi_z,\phi_r) &=  \mathcal{A}(\boldsymbol{\eta};\phi_u,\phi_q,\phi_p,\phi_s,\phi_v,\phi_z,\phi_r) +\mathcal{H}(f;u,u_h,\phi_r) \\&\qquad+ \mathcal{Q}(u,q;u_h,q_h;\phi_p,\phi_r),
\end{align}
where 
\begin{equation}\label{B_ij}
\begin{split}
  \mathcal{B}(\boldsymbol{\xi};&\phi_u,\phi_q,\phi_p,\phi_s,\phi_v,\phi_z,\phi_r)= (\rho^q,\phi_q)_{\mathcal T_h}
+(\xi^u,\partial_x\phi_q)_{\mathcal T_h}
-\langle \widehat \xi^u ,n_x\phi_q\rangle_{\partial\mathcal T_h}  
 +(\rho^p,\phi_p)_{\mathcal T_h}\\&
+(\rho^s,\phi_s)_{\mathcal T_h}
+(\xi^u,\partial_y\phi_s)_{\mathcal T_h}
-\langle \widehat\xi^u, n_y\phi_s\rangle_{\partial\mathcal T_h} +(\rho^s,\phi_v)_{\mathcal T_h}
+(\xi^v,\partial_x\phi_v)_{\mathcal T_h}
-\langle \widehat\xi^v ,n_x\phi_v\rangle_{\partial\mathcal T_h} \\&+(\rho^z,\phi_z)_{\mathcal T_h}
+(\rho^r,\partial_x\phi_z)_{\mathcal T_h}
-\langle \widehat{\rho^r} ,n_x\phi_z\rangle_{\partial\mathcal T_h} 
+(\rho^r,\phi_r)_{\mathcal T_h}
+(\rho^z,\partial_x\phi_r)_{\mathcal T_h}-(\rho^p,\partial_x\phi_r)_{\mathcal T_h} \\&+(\xi^v,\partial_y\phi_r)_{\mathcal T_h}
-\Bigl\langle
  \widehat\rho^z-\widehat\rho^p,
  n_x\phi_r\Bigr\rangle_{\partial\mathcal T_h}
-\langle \widehat\xi^v, n_y\phi_r\rangle_{\partial\mathcal T_h} +(\xi^u_t,\phi_u)_{\mathcal T_h}
+(\rho^r,\phi_u)_{\mathcal T_h}\\&+  ((\rho^q_t,\rho^q)_{\mathcal T_h}
+(\xi^u_t,\partial_x\rho^q)_{\mathcal T_h}
-\langle \widehat\xi^u_t ,n_x\rho^q\rangle_{\partial\mathcal T_h},
\end{split}
\end{equation}
\begin{equation}\label{A_ij}
\begin{split}
   \mathcal{A}(\boldsymbol{\eta};\phi_u,\phi_q,\phi_p,\phi_s,\phi_v,\phi_z,\phi_r)=((\eta^u)_t,\phi_u)_{\mathcal T_h},
\end{split}
\end{equation}
\begin{equation}\label{H_ij}
    \begin{split}
        \mathcal{H}(f;u,u_h,\phi_r)&= -(f(u)-f(u_h),\partial_x\phi_r)_{\mathcal T_h}+\langle f(u)-\widehat{f(u_h)},n_x\phi_r\rangle_{\partial\mathcal T_h},
    \end{split}
\end{equation}
and
\begin{equation}\label{Q_ij}
    \begin{aligned}
        \mathcal{Q}(u,q;u_h,q_h;\phi_p,\phi_r)&= -(uq-u_hq_h,\partial_x\phi_p)_{\mathcal T_h}+\langle uq-\widehat{u_hq_h},n_x\phi_p\rangle_{\partial\mathcal T_h} -\bigl( \tfrac12(q^{2}-q_h^{2}),
          \partial_x\phi_r\bigr)_{\mathcal T_h}\\&\qquad+\Bigl\langle\tfrac12(q^{2}-\widehat{q_h^{2}}),
  n_x\phi_r\Bigr\rangle_{\partial\mathcal T_h}.
    \end{aligned}
\end{equation}
 Since $\widehat{\xi}^\omega = \widehat{\varepsilon}_\omega$, for  $\omega \in \{u,v\}$,  and $\widehat\rho^q  = \widehat \varepsilon_q$, the system \eqref{eq:error-equationsxi} is endowed with the following boundary conditions
\begin{equation}\label{boundary_data2}
\begin{aligned}
  \widehat \xi^{u}=0~\text{ on } \partial\Omega/\Gamma_T,\quad
  \widehat \rho^q= 0 ~\text{ on }
    \Gamma_L\cup\Gamma_R ,\quad \widehat \xi^v = 0 ~\text{ on } \Gamma_R\cup\Gamma_T.
\end{aligned}
\end{equation}
 Moreover, from \eqref{eq:weak_TCe} and \eqref{eq:primary_flux_errors}, we also have the following transmission condition
\begin{equation}\label{eq:weak_TCe_xi}
\begin{aligned}
&\langle (uq-\widehat{u_hq_h})n_x,  \gamma \rangle_{\partial\mathcal T_h} = 0, \qquad \langle \widehat \xi^un_y, \theta  \rangle_{\partial\mathcal T_h} = 0,\qquad \langle \widehat \xi^v n_x, 
     \lambda  \rangle_{\partial\mathcal T_h} = 0,\\&   
\langle  (\widehat \rho^z-\widehat \rho^p+ (f(u) -\widehat{f(u_h)}))n_x,  \nu \rangle_{\partial\mathcal T_h} =0,\qquad \langle \widehat \xi^v n_y, 
     \zeta  \rangle_{\partial\mathcal T_h} = 0, 
\end{aligned}
\end{equation}
for all test traces \(\gamma, \nu \in \mathcal V^k_h(0)\), \(\theta\in\widetilde M_h^k(\mathcal H\setminus \Gamma_B)\), \(\lambda\in\widetilde M_h^k(\mathcal V\setminus \Gamma_L)\), and \(\zeta \in \widetilde M_h^k( \mathcal H\setminus \Gamma_T)\). By the definition of $\widehat u_h$ and $\widehat q_h$, we have $\langle\widehat \xi^u,\delta n_x\rangle_{\partial\mathcal T_h} = 0$  and $\langle\widehat \rho^q,\sigma n_x\rangle_{\partial\mathcal T_h} = 0$, for all $\delta, \sigma \in \mathcal V^k_h(0)$.

\begin{assumption}[An {\it a priori } assumption ]\label{Aprioriass}In order to establish the $\mathcal O(h^{k+1/2})$ order of convergence and deal with the nonlinear term, we require the following assumptions for $t< T$.
    We assume 
    \begin{equation}\label{ass1uu_h}
        \begin{aligned}
            \|u-u_h\|_{\mathcal T_h} \leq h^{2}, \qquad \|q-q_h\|_{\mathcal T_h} \leq h^{2}.
        \end{aligned}
    \end{equation}
    Consequently, we have the following
     \begin{equation}\label{ass1uu_h2}
        \begin{aligned}
            &\|u-u_h\|_{L^{\infty}(\mathcal T_h)} \leq h, \qquad \|u-u_h\|_{L^{\infty}(\partial \mathcal T_h)} \leq h^{1/2}, \\&  \|q-q_h\|_{L^{\infty}(\mathcal T_h)} \leq h, \qquad \|q-q_h\|_{L^{\infty}(\partial \mathcal T_h)} \leq h^{1/2}.
        \end{aligned}
    \end{equation}
\end{assumption}

\begin{theorem}[A-priori $L^{2}$ error bound]\label{thm:error_semi}
Assume that the CH–KP problem \eqref{eq:ibvp} admits a classical solution
\(
    u\in C\bigl([0,T];X^{k+3}(\Omega)\bigr)\cap C^{1}\bigl([0,T];X^{k+1}(\Omega)\bigr)
\) with homogeneous boundary conditions and $u_0 \in X^{k+3}(\Omega)$.
Let \((u_h,q_h)\in V_h^{k}\times V_h^{k}\) with
\(k\ge2\) be the semi-discrete HDG approximation produced by the scheme
\eqref{eq:hdg-local}-\eqref{boundary_data} on a regular Cartesian mesh \(\mathcal T_h\). Assume that stabilization parameters $\tau $s satisfy Assumption \ref{globalass}.
Then, for \(h\) sufficiently small and for every \(t\in[0,T]\),
\begin{equation}\label{eq:err_bound}
    \begin{aligned}
        \|u(t)-u_h(t)\|_{\mathcal T_h}
        +\|q(t)-q_h(t)\|_{\mathcal T_h}\le Ch^{k+1/2},
    \end{aligned}
\end{equation}
where the constant \(C>0\) depends only on the final time \(T\), the
polynomial degree \(k\), the Sobolev norm
\(
  \displaystyle
  \max_{0\le\tau\le T}\|u(\tau)\|_{H^{k+2}(\Omega)},
\) constant $\delta>0$,
and an upper bound for the first and second derivatives of the nonlinear flux
\(f(u)\).
\end{theorem}

\begin{proof}
We choose  $(\phi_u,\phi_q,\phi_p,\phi_s,\phi_v,\phi_z,\phi_r) = (\xi^u-\partial_x\rho^q,\rho^p-\rho^z,-\rho^q,-\xi^v,\xi^v,\rho^q,-\xi^u)$ as a test function in \eqref{err_eqn_cpt_xi}.
This gives
\begin{align}\label{err_eqn_cpt_test}
    \nonumber\mathcal{B}(\boldsymbol{\xi};\xi^u-\partial_x\rho^q,\rho^p-\rho^z,-\rho^q,-\xi^v,\xi^v,\rho^q,-\xi^u) &=  \mathcal{A}(\boldsymbol{\eta};\xi^u-\partial_x\rho^q,\rho^p-\rho^z,-\rho^q,-\xi^v,\xi^v,\rho^q,-\xi^u) \\&\quad+\mathcal{H}(f;u,u_h,-\xi^u) + \mathcal{Q}(u,q;u_h,q_h;-\rho^q,-\xi^u).
\end{align}
We first evaluate the left-hand side of
\eqref{err_eqn_cpt_test} with tailor-made test functions and perform some algebraic manipulations.
 \begin{equation*}
        \begin{split}
   &     \mathcal{B}(\boldsymbol{\xi};\xi^u-\partial_x\rho^q,\rho^p-\rho^z,-\rho^q,-\xi^v,\xi^v,\rho^q,-\xi^u) \\& =     \big(\xi^u_t,\xi^u\big)_{\mathcal T_h}+\big(\rho^q_t,\rho^q\big)_{\mathcal T_h}
      -\langle\widehat \xi^u_t ,n_x\rho^q\rangle_{\partial\mathcal T_h} 
  +(\xi^u,\partial_x\rho^p)_{\mathcal T_h}-\langle \widehat \xi^u ,n_x\rho^p\rangle_{\partial \mathcal T_h}- (\xi^u,\partial_x\rho^z)_{\mathcal T_h} \\
  &\quad+\langle \widehat \xi^u,n_x\rho^z\rangle_{\partial \mathcal T_h} 
  - (\xi^u,\partial_y\xi^v)_{\mathcal T_h}+\langle \widehat \xi^u,n_y\xi^v\rangle_{\partial \mathcal T_h}
  + (\xi^v,\partial_x \xi^v)_{\mathcal T_h}
  -\langle \widehat \xi^v,n_x\xi^v\rangle_{\partial \mathcal T_h} - (\xi^v,\partial_y\xi^u)_{\mathcal T_h}
  \\
  &\quad+\langle \widehat \xi^v,n_y\xi^u\rangle_{\partial \mathcal T_h}
  -\langle \widehat{\rho^r},n_x\rho^q\rangle_{\partial \mathcal T_h} 
  - (\rho^z,\partial_xu_h)_{\mathcal T_h}
  +\bigl(\rho^p,\partial_x\xi^u\bigr)_{ \mathcal T_h}
  +\langle \widehat \rho^z
             -\widehat \rho^p,n_x\xi^u
             \rangle_{\partial \mathcal T_h}.
\end{split}
\end{equation*}
Next, we use multiple integration by parts to obtain the following identity:
\begin{equation*}
\begin{split}
&     \mathcal{B}(\boldsymbol{\xi};\xi^u-\partial_x\rho^q,\rho^p-\rho^z,-\rho^q,-\xi^v,\xi^v,\rho^q,-\xi^u)\\& =     \big(\xi^u_t,\xi^u\big)_{\mathcal T_h}+\big(\rho^q_t,\rho^q\big)_{\mathcal T_h}
      -\langle\widehat \xi^u_t n_x,\rho^q\rangle_{\partial\mathcal T_h} 
  +\langle\xi^u,n_x\rho^p\rangle_{\partial\mathcal T_h}-\langle \widehat \xi^u n_x,\rho^p\rangle_{\partial \mathcal T_h}- \langle\xi^u,n_x\rho^z\rangle_{\partial\mathcal T_h} \\
  &\quad+\langle \widehat \xi^un_x,\rho^z\rangle_{\partial \mathcal T_h} 
  - \langle \xi^u,n_y\xi^v\rangle_{\partial\mathcal T_h}+\langle \widehat \xi^un_y,\xi^v\rangle_{\partial \mathcal T_h}
  + \frac{1}{2}\langle\xi^v,n_x \xi^v\rangle_{\partial\mathcal T_h}
  -\langle \widehat \xi^vn_x,\xi^v\rangle_{\partial \mathcal T_h}
  \\
  &\quad+\langle \widehat \xi^vn_y,\xi^u\rangle_{\partial \mathcal T_h}
  -\langle \widehat{\rho^r}n_x,\rho^q\rangle_{\partial \mathcal T_h} 
  +\langle \widehat \rho^z
             -\widehat \rho^p,n_x\xi^u
             \rangle_{\partial \mathcal T_h}.
\end{split}
\end{equation*}
The transmission condition
\eqref{eq:weak_TCe_xi} and boundary condition \eqref{boundary_data2} yield the following identity:
\begin{equation}\label{B_ij_est}
\begin{split}
&     \mathcal{B}(\boldsymbol{\xi};\xi^u-\partial_x\rho^q,\rho^p-\rho^z,-\rho^q,-\xi^v,\xi^v,\rho^q,-\xi^u)\\& =     \big(\xi^u_t,\xi^u\big)_{\mathcal T_h}+\big(\rho^q_t,\rho^q\big)_{\mathcal T_h}
 -\langle \widehat \xi^u -\xi^u,n_x((\widehat \rho^z -\widehat \rho^p)-(\rho^z-\rho^p))\rangle_{\partial \mathcal T_h} -\langle \widehat \xi^u-\xi^u,n_y(\widehat \xi^v-\xi^v)\rangle_{\partial \mathcal T_h} \\
  &\quad
  + \Bigl\langle \frac{1}{2} (\widehat \xi^v-\xi^v)^2, n_x\Bigl\rangle_{\partial \mathcal T_h} +\frac{1}{2}(\widehat \xi^v) ^2|_{\Gamma_L} -  \langle (f(u)-\widehat{f(u_h)}),n_x\widehat \xi^u
             \rangle_{\partial \mathcal T_h}. 
        \end{split}
    \end{equation}
Eventually, the numerical-flux errors \eqref{eq:vertical_left_flux}–\eqref{eq:horizontal_flux} yield 
    \begin{equation}\label{B_ij_est3}
\begin{split}
&    \mathcal{B}(\boldsymbol{\xi};\xi^u-\partial_x\rho^q,\rho^p-\rho^z,-\rho^q,-\xi^v,\xi^v,\rho^q,-\xi^u)\\
  & = \big(\xi^u_t,\xi^u\big)_{\mathcal T_h}+\big(\rho^q_t,\rho^q\big)_{\mathcal T_h}
 -\langle \tau_{zpu}^+, (\widehat \xi^u - \xi^u)^2\rangle_{L}
 + \langle\mu^z-\mu^p - \tau_{zpu}^{+}\eta^u n_x, (\widehat \xi^u - \xi^u)n_x\rangle_{L} \\& \quad
 -\langle \tau_{zpu}^-, (\widehat \xi^u - \xi^u)^2\rangle_{R}-\langle \tau_{zpv}^-, (\widehat \xi^u - \xi^u)((\widehat \xi^v)^- - \xi^v)\rangle_{R}\\&\quad + \langle\mu^z-\mu^p - \tau_{zpu}^{-}\eta^u n_x- \tau_{zpv}^{-}\eta^v n_x, (\widehat \xi^u - \xi^u)n_x\rangle_{R}\\&\quad+ \Bigl\langle \frac{1}{2}, ((\widehat \xi^v)^--\xi^v)^2\Bigl\rangle_{R} 
   +\frac{1}{2}(\widehat \xi^v) ^2|_{\Gamma_L} -  \langle (f(u)-\widehat{f(u_h)}),n_x\widehat \xi^u
             \rangle_{\partial \mathcal T_h}. 
        \end{split}
    \end{equation}
    Finally, employing the Assumption \ref{globalass} and ignoring the positive term $\frac{1}{2}(\widehat \xi^v) ^2|_{\Gamma_L}$ to get the first inequality, and then utilizing the Young's inequality and projection estimates \eqref{eq:proj-est}-\eqref{mod_proj_est} along with the smoothness assumption on exact variables $u,p,v,z$ to obtain the second inequality as follows 
 \begin{equation}\label{B_ij_est4}
\begin{aligned}
& \mathcal{B}(\boldsymbol{\xi};\xi^u-\partial_x\rho^q,\rho^p-\rho^z,-\rho^q,-\xi^v,\xi^v,\rho^q,-\xi^u) \\&\geq  \frac{1}{2}\frac{d}{dt}\|\xi^u\|_{\mathcal T_h}^2 +\frac{1}{2}\frac{d}{dt}\|\rho^q\|_{\mathcal T_h}^2-  \langle (f(u)-\widehat{f(u_h)}),n_x\widehat \xi^u
             \rangle_{\partial \mathcal T_h}  + \langle\mu^z-\mu^p - \tau_{zpu}^{+}\eta^u n_x, (\widehat \xi^u - \xi^u)n_x\rangle_{L} \\&\quad+ \langle\mu^z-\mu^p - \tau_{zpu}^{-}\eta^u n_x- \tau_{zpv}^{-}\eta^v n_x, (\widehat \xi^u - \xi^u)n_x\rangle_{R}
              \\&\geq  \frac{1}{2}\frac{d}{dt}\|\xi^u\|_{\mathcal T_h}^2 +\frac{1}{2}\frac{d}{dt}\|\rho^q\|_{\mathcal T_h}^2 -  \langle (f(u)-\widehat{f(u_h)}),n_x\widehat \xi^u
             \rangle_{\partial \mathcal T_h}   - C(\delta)h^{2k+1}- C\delta\|\widehat \xi^u - \xi^u\|^2_{\mathcal V}, 
        \end{aligned}
    \end{equation} 
    where generic constant $C>0$ and $\delta >0$ is small, comes from the application of Young's inequality.
Next estimate $\mathcal A$, by incorporating the Young's inequality, $L^{2}$-projection property \eqref{mod_proj_est} along with the fact $u\in$ $ C^{1}\bigl([0,T];X^{k+1}(\Omega)\bigr) $, we have 
\begin{equation}\label{A_ij_est}
\begin{aligned}
     \mathcal{A}(\boldsymbol{\eta};\xi^u-\partial_x\rho^q,\rho^p-\rho^z,-\rho^q,-\xi^v,\xi^v,\rho^q,-\xi^u)  
\leq  C\|\xi^u\|_{\mathcal T_h}^2 + Ch^{2k+2}.
\end{aligned}
\end{equation}
Substituting all the estimates for $\mathcal{B}$, $\mathcal{A}$, $\mathcal{Q}$, and $\mathcal{H}$ from \eqref{B_ij_est4}, \eqref{A_ij_est}, \eqref{Q_ij_est}, and \eqref{H_ij_est} respectively in \eqref{err_eqn_cpt_test} along with the assumption \eqref{ass3} and small enough $\delta$ such that $\bar C -C\delta\geq0$, we end up with
\begin{equation}\label{complete_term}
    \begin{aligned}
        \frac{1}{2}\frac{d}{dt}\bigl(\|\xi^u\|^2_{\mathcal T_h}+\|\rho^q\|^2_{\mathcal T_h}\bigr) + (\bar C-C\delta)\|\widehat \xi^u -\xi^u\|_{ \mathcal V}^2  \leq   Ch^{2k+1} + C(\|\xi^u\|^2_{\mathcal T_h}+\|\rho^q\|^2_{\mathcal T_h}), 
    \end{aligned}
\end{equation}
where the constant \(C>0\) depends only on the final time \(T\), the
polynomial degree \(k\), the Sobolev norm
\(
  \displaystyle
  \max_{0\le\tau\le T}\|u(\tau)\|_{H^{k+2}(\Omega)},
\)
constant $\delta$, and an upper bound for the first and second derivatives of the nonlinear flux
\(f(u)\). Discarding the positive terms from the left hand side and applying the Gr\"onwall’s
inequality gives
\begin{equation}\label{err_est_xi_rho}
    \|\xi^u\|^2_{\mathcal T_h}+\|\rho^q\|^2_{\mathcal T_h} \leq Ch^{2k+1}.
\end{equation}
Combining \eqref{err_est_xi_rho} and the projection estimates \eqref{eq:proj-est} and \eqref{mod_proj_est}, we finally obtain
\begin{equation}\label{eq:err_bound2}
    \begin{aligned}
         \|u(t)-u_h(t)\|_{\mathcal T_h}+\|q(t)-q_h(t)\|_{\mathcal T_h}\leq Ch^{k+1/2},
    \end{aligned}
\end{equation}
 for all $t\in[0,T]$.
This completes the proof.
\end{proof}

The justification of assumption \eqref{ass1uu_h} and \eqref{ass1uu_h2} follows via a continuity argument \cite{XuShu2007}. For $k \geq 2$ and $h$ sufficiently small, $Ch^{k+1/2} < \frac{1}{2}h^2$, where  $C$ is constant in \eqref{eq:err_bound} can be exactly determined by using the final time $T$. Define 
\(
t^* := \sup \left\{ t : \|u(t) - u_h(t)\|_{\mathcal{T}_h} \leq h^{2}\right\}.
\)
Initial conditions satisfy $t^* > 0$ by projection properties. Assume $t^* < T$. Continuity implies 
\(
\|u(t^*) - u_h(t^*)\|_{\mathcal{T}_h} = h^{2}.
\)
Theorem \ref{thm:error_semi} (valid for $t \leq t^*$ under the assumption) yields
\[
\|u(t^*) - u_h(t^*)\|_{\mathcal{T}_h} \leq Ch^{k+1/2} < \frac{1}{2}h^{2}.
\]
This contradicts the equality if $t^*<T$, and hence $t^* \geq T$. The estimate  $\|q-q_h\|_{\mathcal T_h} \leq h^{2}$ follows similarly. The consequences in \eqref{ass1uu_h2} follow from inverse and interpolation inequalities \eqref{invesre}, and \eqref{ass1uu_h}. 

\section{Numerical Experiments}\label{sec:numerics}

This section presents a numerical experiments to validate the accuracy and demonstrate the robustness of the proposed HDG method for the CH-KP equation \eqref{eq:ibvp}. All simulations are performed on Cartesian meshes using tensor-product polynomial spaces $Q_k(K_{ij})$ with polynomial degrees $k = 1, 2, 3$ (linear, quadratic, and cubic). The semi-discrete HDG scheme is advanced in time by the implicit Euler method with a fixed time step $\Delta t = 0.001$, chosen sufficiently small so that temporal errors are negligible compared to the spatial discretization errors. 

For each test case, boundary conditions and right-hand side terms for the nonlocal operator $\partial_x^{-1} u_y$ are derived from the corresponding exact solution. We note that the method remains stable and accurate even on relatively coarse grids, i.e., for a small number of elements $N_x \times N_y$. In the convergence studies, we consider a sequence of uniformly refined meshes with $N_x = N_y = 2, 4, 8, \dots, 64$ and define the mesh size as $h = 1/N_x$ for domain $\Omega$. 
The stabilization parameters appearing in the numerical fluxes \eqref{Vertical_trace}--\eqref{Horizontal_trace} are selected in accordance with Assumption \ref{globalass}. In particular, for all the experiments reported below we set
\[
 \tau^+_{zpu} = \tau^-_{zpu}= -1, \qquad \tau^-_{zpv} = 1, \qquad \tau_{uqq} = -\frac{1}{4}, \qquad \text{ and } \quad \tau_f(\hat{u}_h, u_h) = 4.
\]
These choices satisfy the conditions of Assumption \ref{globalass} and were observed to yield optimal convergence rates and stable computations. The implementation of the HDG scheme, including the assembly of the global system and the local solvers, is available in the open-source repositories \cite{RuppGK22,HyperHDGgithub}.

\subsection{Smooth solution: Accuracy test}\label{SEC:smooth_test}
Let the manufactured solution be chosen as
\[
 u(x,y,t) = e^{-t} \sin (x) \sin (y) \qquad \text{ for } t \in [0,1] \text{ and } x,y \in (0, 2\pi).
\]
To ensure consistency, an artificial source term is computed analytically and inserted into the CH-KP equation \eqref{eq:ibvp} with parameter $\kappa = -1/2$. 
We evaluate the \( L^2 \) errors of \( u_h \) and \( q_h \), approximating exact solution $u$ and $u_x$ respectively, at $T=1$ and display them in Figure \ref{FIG:error_plots}. All plots indicate optimal orders of convergence \( \mathcal O(h^{k+1}) \) in both unknowns.
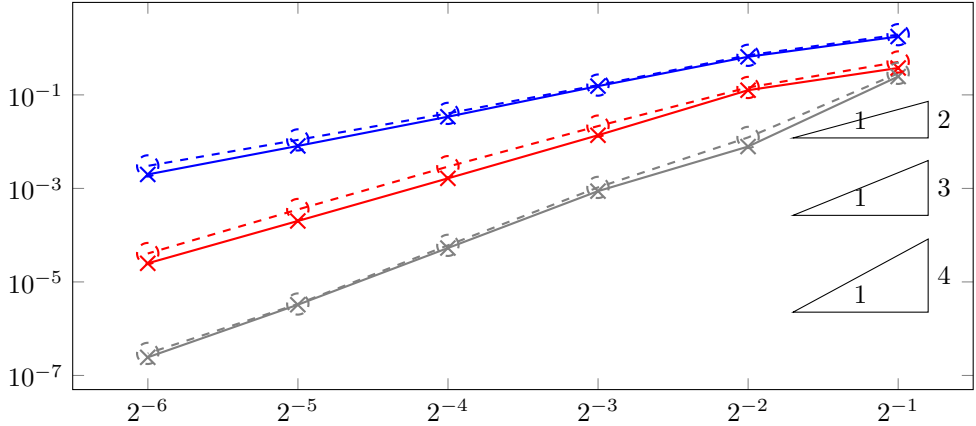
\begin{figure}[!ht]
\begin{tikzpicture}
 \begin{loglogaxis}[width=.9\textwidth, height=.45\textwidth, log basis x={2}]
  \addplot[thick,color=blue, mark size=4,mark=x] table {
   0.5       1.77
   0.25      6.53e-1
   0.125     1.54e-1
   0.0625    3.41e-2
   0.03125   8.09e-3
   0.015625  1.99e-3
  };
  \addplot[thick,color=red, mark size=4,mark=x] table {
   0.5       3.76e-1
   0.25      1.26e-1
   0.125     1.37e-2
   0.0625    1.64e-3
   0.03125   2.01e-4
   0.015625  2.5e-05
  };
  \addplot[thick,color=gray, mark size=4,mark=x] table {
   0.5       2.47e-01
   0.25      7.85e-3
   0.125     8.74e-4
   0.0625    5.31e-5
   0.03125   3.25e-6 
  0.015625   2.44e-7
  };

  \addplot[thick,color=blue, mark size=4, mark=o, dashed] table {
   0.5       1.94
   0.25      7.03e-1
   0.125     1.62e-1
   0.0625    4.02e-2
   0.03125   1.09e-2
   0.015625  3.03e-3
  };
  \addplot[thick,color=red, mark size=4, mark=o, dashed] table {
   0.5       5.06e-1
   0.25      1.44e-1
   0.125     2.15e-2
   0.0625    2.92e-3
   0.03125   3.54e-4
   0.015625  4.05e-05
  };
  \addplot[thick,color=gray, mark size=4, mark=o, dashed] table {
   0.5       2.89e-1
   0.25      1.23e-2
   0.125     1.05e-3
   0.0625    6.06e-5
   0.03125   3.37e-6
   0.015625  2.96e-7
  };

  \logLogSlopeTriangle{0.95}{0.15}{0.65}{2}{black};
  \logLogSlopeTriangle{0.95}{0.15}{0.45}{3}{black};
  \logLogSlopeTriangle{0.95}{0.15}{0.2}{4}{black};
 \end{loglogaxis}
\end{tikzpicture}
\caption{Error plots for \(u\) (solid lines) and \(q\) (dashed lines) from the accuracy test described in Section~\ref{SEC:smooth_test}. The abscissa shows the mesh size, while the ordinate displays the \(L^2\) error.  Blue curves correspond to linear shape functions, red curves to quadratic shape functions, and gray curves to cubic shape functions.}\label{FIG:error_plots}
\end{figure}

\begin{figure}[ht]
    \centering
    \begin{subfigure}[b]{0.49\textwidth}
        \includegraphics{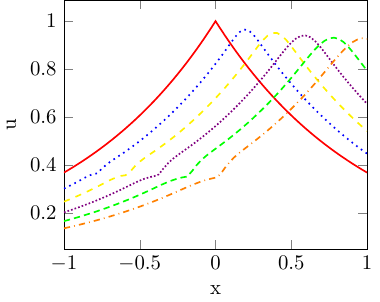}
        \caption{Cross-section $x=0$}
        \label{fig:figure1_1d01}
    \end{subfigure}
    \hfill
    \begin{subfigure}[b]{0.49\textwidth}
        \includegraphics{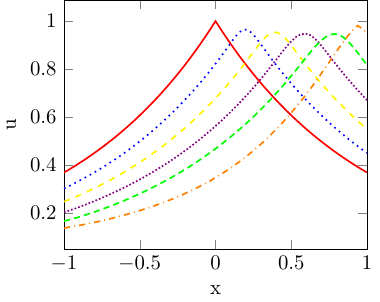}
        \caption{Cross-section $y=0$}
        \label{fig:figure2_1d01}
    \end{subfigure}
    \caption{Cross-sectional profiles of the peakon solution using $\Delta t = 0.01$. From left to right, each curve corresponds to time $t = 0, 0.2, 0.4, 0.6, 0.8, 1.0$. Minor oscillations appear near the peak due to the larger time step, but the overall shape and location remain well captured.}
    \label{fig:both_figures_1d01}
\end{figure}

\begin{figure}[ht]
    \centering
    \begin{subfigure}[b]{0.49\textwidth}
        \includegraphics{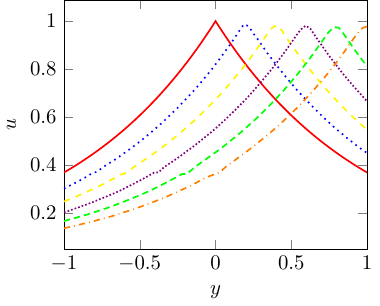}
        \caption{Cross-section $x=0$}
        \label{fig:figure1_1d}
    \end{subfigure}
    \hfill
    \begin{subfigure}[b]{0.49\textwidth}
        \includegraphics{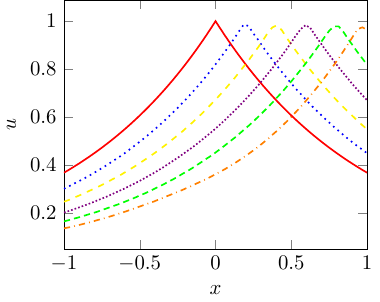}
        \caption{Cross-section $y=0$}
        \label{fig:figure2_1d}
    \end{subfigure}
    \caption{Cross-sectional profiles of the peakon solution using $\Delta t = 0.001$. From left to right, each curve corresponds to time $t = 0, 0.2, 0.4, 0.6, 0.8, 1.0$. The smaller time step effectively suppresses oscillations and yields a sharper, more accurate representation of the peak.}
    \label{fig:both_figures_1d}
\end{figure}

\begin{figure}[ht]
    \centering
    \begin{subfigure}[b]{0.49\textwidth}
        \includegraphics{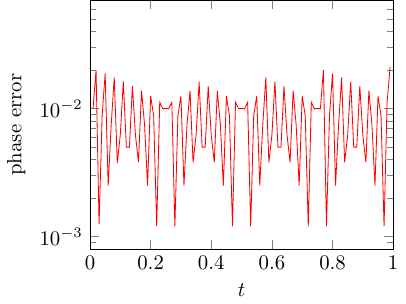}
        \caption{Cross-section $x=0$}
    \end{subfigure}
    \hfill
    \begin{subfigure}[b]{0.49\textwidth}
        \includegraphics{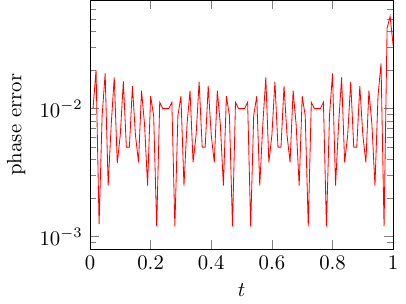}
        \caption{Cross-section $y=0$}
    \end{subfigure}
    \caption{Phase error for the peakon solution. The left and right panels show the error in the cross-sectional profiles along $x=0$ and $y=0$, respectively. The HDG method with $Q_2$ elements on a $64\times64$ grid and $\Delta t=0.001$ is used.}
    \label{fig:both_figures_phaseerror}
\end{figure}

\begin{figure}[ht]
    \centering
    \begin{subfigure}[b]{0.49\textwidth}
        \centering
        \includegraphics[width=\linewidth]{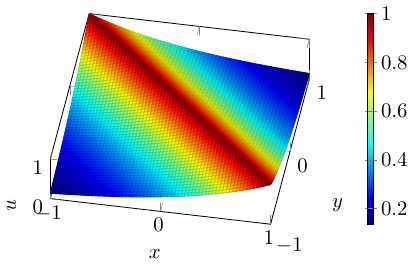}
        \vspace{2pt}
        \subcaption{t=0}
        \label{fig:peakon_t0}
    \end{subfigure}
    \hfill
    \begin{subfigure}[b]{0.49\textwidth}
        \centering
        \includegraphics[width=\linewidth]{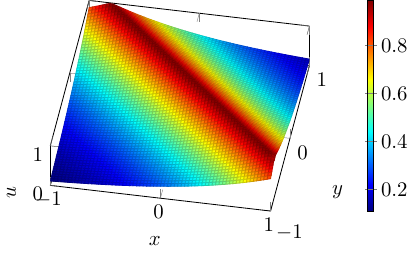}
        \vspace{2pt}
        \subcaption{t=0.2}
        \label{fig:peakon_t02}
    \end{subfigure}
    
    \vspace{0.5cm}
    
    \begin{subfigure}[b]{0.49\textwidth}
        \centering
        \includegraphics[width=\linewidth]{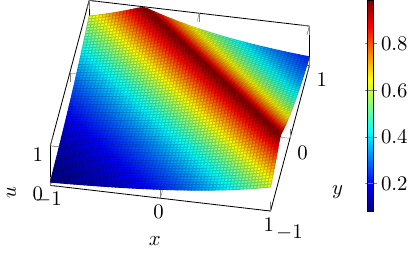}
        \vspace{2pt}
        \subcaption{t=0.5}
        \label{fig:peakon_t05}
    \end{subfigure}
    \hfill
    \begin{subfigure}[b]{0.49\textwidth}
        \centering
        \includegraphics[width=\linewidth]{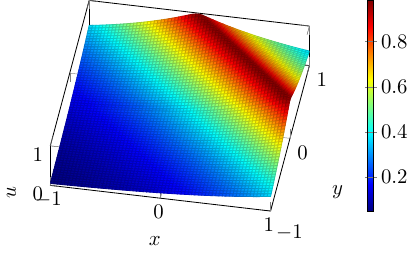}
        \vspace{2pt}
        \subcaption{t=1}
        \label{fig:peakon_t10}
    \end{subfigure}
    \caption{Three-dimensional evolution of the peakon solution over time, computed with $\Delta t = 0.001$.}
    \label{fig:peakon_evolution_grid}
\end{figure}

\subsection{Peakon solution}
Next, we solve the peakon solution \cite{ZhangXuLiu2023}
\[
 u(x,y,t) = c\exp(-|x + y - ct|) \qquad \text{ for } t \in [0,4] \text{ and } x,y \in (0,1)
\]
with parameter $\kappa = -1/2$ and $c=1$ on $[-1,1] \times [-1,1]$ under Dirichlet boundary conditions. The HDG method resolves the propagating peak with minimal dispersion using $Q_2$ elements on $64\times 64$ grids. Cross-sectional profiles confirm accurate peak localization. 

The peakon solution is a non-smooth, peaked solitary wave that arises in the Camassa--Holm family of equations \cite{CamassaHolm1993,Pelinovsky2024,GuiLiuLuoYin2021,HoldenRaynaud2006}. Physically, it models finite-time wave breaking in shallow water waves and retains its shape while propagating. 

Figure \ref{fig:both_figures_1d01} displays cross-sectional profiles of the peakon solution computed with a time step $\Delta t = 0.01$. The left panel shows the cross-section along $x=0$, while the right panel shows $y=0$. Although the overall shape and propagation speed are well represented, small oscillations are visible near the peak due to the relatively coarse temporal discretization.
Reducing the time step to $\Delta t = 0.001$ (Figure \ref{fig:both_figures_1d}) significantly improves the solution quality. The oscillations are effectively suppressed, and the peak profile becomes sharper and more accurate. This demonstrates the temporal convergence of the scheme and its ability to resolve steep gradients with appropriate time-step refinement. The phase error of the peakon solution, depicted in Figure \ref{fig:both_figures_phaseerror}, remains minimal, demonstrating the method's capability to accurately capture the wave's propagation without significant distortion.
The three-dimensional evolution of the peakon is illustrated in Figure \ref{fig:peakon_evolution_grid} at four time instances ($t = 0, 0.2, 0.5, 1$). The solution propagates without noticeable distortion or dispersion, confirming that the HDG method preserves the shape and velocity of the peakon over time.

The proposed HDG method \eqref{eq:hdg-local}-\eqref{boundary_data} demonstrates optimal accuracy for smooth solutions of the CH--KP equation. Even for challenging peakon solutions, the scheme resolves the peaks accurately with minimal numerical dissipation and phase errors that caused by the temporal discretization. The method is computationally efficient due to its hybridized structure, which reduces globally coupled degrees of freedom, and it achieves optimal results even on relatively coarse Cartesian grids. These attributes make the HDG method a powerful tool for simulating nonlinear dispersive wave equations in two dimensions.

\section{Conclusion}\label{sec:conclusion}
This paper has introduced a novel high-order hybridizable discontinuous Galerkin (HDG) method for solving the two-dimensional Camassa-Holm-Kadomtsev-Petviashvili (CH-KP) equation by leveraging Cartesian meshes with tensor-product polynomial spaces \(Q_k(K_{ij}) = P_k(I_i) \otimes P_k(J_j)\). The semi-discrete scheme maintains energy stability (\(\frac{d}{dt}\mathcal{E}_h(t) \leq 0\)).  
 Error analysis establishes \(\mathcal{O}(h^{k+1/2})\) convergence in the energy norm for smooth solutions achieved via a special projection operator.  Hybridization reduces globally coupled degrees of freedom in local DG methods, enabling scalable simulations.  Numerical experiments validate the method's capability to accurately capture smooth solutions and peakon solutions.

Future work will extend this direct {\it non-integration} HDG framework to {\it integration} HDG approach by integrating the equation $v_x =u_y$ and setting  
\[
v_h(x,y)\big|_{K_{ij}} = \widehat{v}_h(x_R,y) - \int^{x_R}_{x} s_h(\xi,y)d\xi, \quad s_h = \partial_y u_h,
\]
and develop the Hamiltonian conservative scheme. 

The HDG framework presented here is readily extensible to other nonlocal, quasilinear dispersive wave equations—including the Camassa–Holm, Degasperis–Procesi, 
Novikov, and various Kadomtsev–Petviashvili-type systems—as well as to related geophysical fluid models, offering a versatile and high‑order accurate computational  tool for a wide range of applications.

\bibliographystyle{abbrv}
\bibliography{Main}
\appendix
\section{Appendix: Proof of the technical lemmas}\label{sec:proofs}

\begin{lemma}\label{Q_lemma}
    Let $\mathcal{Q}$ be defined by \eqref{Q_ij}. Assume that all the hypotheses of Theorem \ref{thm:error_semi} hold. Then we have the following estimate
   \begin{equation}\label{Q_ij_est}
\begin{split}
    \mathcal{Q}(u,q;u_h,q_h;-\rho^q,-\xi^u)\leq Ch^{2k+1} + C (\|\rho^q\|^2_{\mathcal T_h} + \|\xi^u\|_{\mathcal T_h}^2).
     \end{split}
\end{equation} 
\end{lemma}
\begin{proof}
We estimate the term $\mathcal Q$ on the right hand side of \eqref{err_eqn_cpt_test}. Using the definition of $\mathcal Q$ by \eqref{Q_ij} and incorporating \eqref{Vertical_trace}, and \eqref{eq:weak_TCe_xi}, we obtain
\begin{equation}\label{Q_xi}
    \begin{split}
        \mathcal{Q}(u,q;u_h,q_h;-\rho^q,-\xi^u)&= (uq-u_hq_h,\partial_x\rho^q)_{\mathcal T_h}-\langle uq-\widehat{u_hq_h},n_x\rho^q\rangle_{\partial\mathcal T_h} +\bigl( \tfrac12(q^{2}-q_h^{2}),
          \partial_x\xi^u\bigr)_{\mathcal T_h}\\&\quad-\Bigl\langle\tfrac12(q^{2}-\widehat{q_h^{2}}),
  n_x\xi^u\Bigr\rangle_{\partial\mathcal T_h}\\
  &= (uq-u_hq_h,\partial_x\rho^q)_{\mathcal{T}_h}+\Bigl\langle uq-u_h\left(\frac{\widehat q_h+ q_h}{2}\right),n_x(\widehat\rho^q-\rho^q)\Bigr\rangle_{\partial\mathcal T_h}\\&\quad  + \langle\tau_{uqq}( q_h-\widehat q_h), (n_x)^2(\widehat\rho^q-\rho^q)\rangle_{\partial \mathcal T_h} +\bigl( \tfrac12(q^{2}-q_h^{2}),
          \partial_x\xi^u\bigr)_{\mathcal T_h}\\&\quad-\Bigl\langle\tfrac12(q^{2}-\widehat q_h^{2}),
  n_x\xi^u\Bigr\rangle_{\partial\mathcal T_h}.
    \end{split}
\end{equation}
Note that we have the identities
\begin{equation}\label{q2-hatq2}
\begin{aligned}
 \frac{1}{2}q^2 - \frac{1}{2}q_h^2 &= q(q-q_h)-\frac{1}{2}(q-q_h)^2 = q(\rho^q-\mu^q) -\frac{1}{2}(\rho^q-\mu^q)\\ &= q\rho^q -q\mu^q - \frac{1}{2}(\rho^q)^2 +\rho^q\mu^q - \frac{1}{2}(\mu^q)^2,
 \end{aligned}
\end{equation}
\begin{equation}\label{uq-uhqh3}
    \begin{aligned}
        \frac{1}{2}q^2 - \frac{1}{2}\widehat q_h^2 &= q(q-\widehat q_h)- \frac{1}{2}(q-\widehat q_h)^2 = q\widehat \rho^q- \frac{1}{2}(\widehat \rho^q)^2,
    \end{aligned}
\end{equation}
\begin{equation}\label{uq-uhqh}
    \begin{aligned}
        uq-u_hq_h &= u(q-q_h)+q(u-u_h)- (u-u_h)(q-q_h)\\
        &= u\rho^q-u\mu^q+q\xi^u-q\eta^u -\xi^u\rho^q+\xi^u\mu^q +\eta^u\rho^q - \eta^u\mu^q,
    \end{aligned}
\end{equation}
and that
\begin{equation}\label{uq-hat_uhqh}
    \begin{aligned}
        uq- u_h\left(\frac{\widehat q_h +q_h}{2}\right) &= \frac{1}{2}(u(q-\widehat q_h)+q(u- u_h)- (u- u_h)(q-\widehat q_h))\\
        &\quad+ \frac{1}{2}(u(q- q_h)+q(u- u_h)- (u- u_h)(q- q_h))
        \\ &= \frac{1}{2}\bigl(u\widehat\rho^q +q\xi^u - q\eta^u -\xi^u\widehat \rho^q + \eta^u\widehat \rho^q)\\&\quad+\frac{1}{2}\bigl(u(\rho^q-\mu^q)  +q\xi^u - q\eta^u -\xi^u\rho^q +\xi^u\mu^q +\eta^u\rho^q -\eta^u\mu^q\bigr)\\&  =\frac{1}{2}u\bigl(\widehat\rho^q +\rho^q\bigr) +\frac{1}{2}\eta^u\bigl(\widehat \rho^q +\rho^q\bigr) -\frac{1}{2}\xi^u(\widehat \rho^q+\rho^q) +q \xi^u-q\eta^u \\&\quad+\frac{1}{2}\bigl(-u\mu^q  +\xi^u\mu^q-\eta^u\mu^q\bigr).
    \end{aligned}
\end{equation}
Putting values from \eqref{q2-hatq2}-\eqref{uq-hat_uhqh} into the equation \eqref{Q_xi} yields
\begin{equation*}
    \begin{split}
        \mathcal{Q}&(u,q;u_h,q_h;-\rho^q,-\xi^u)= (u\rho^q-u\mu^q+q\xi^u-q\eta^u -\xi^u\rho^q+\xi^u\mu^q +\eta^u\rho^q - \eta^u\mu^q,\partial_x\rho^q)_{\mathcal{T}_h}\\&\quad+\Bigl\langle\frac{1}{2}u\bigl(\widehat\rho^q +\rho^q\bigr) +\frac{1}{2}\eta^u\bigl(\widehat \rho^q +\rho^q\bigr) -\frac{1}{2}\xi^u(\widehat \rho^q+\rho^q) +q \xi^u-q\eta^u , n_x(\widehat\rho^q-\rho^q)\Bigr\rangle_{\partial\mathcal T_h}\\&\quad  +\Bigl\langle\frac{1}{2}\bigl(-u\mu^q  +\xi^u\mu^q-\eta^u\mu^q\bigr),n_x(\widehat\rho^q-\rho^q)\Bigr\rangle_{\partial\mathcal T_h}+ \langle\tau_{uqq}(\widehat\rho^q-\rho^q+\mu^q),(n_x)^2 (\widehat\rho^q-\rho^q)\rangle_{\partial \mathcal T_h}\\&\quad +\bigl( q\rho^q -q\mu^q - \frac{1}{2}(\rho^q)^2 +\rho^q\mu^q - \frac{1}{2}(\mu^q)^2,
          \partial_x\xi^u\bigr)_{\mathcal T_h}-\Bigl\langle q\widehat \rho^q- \frac{1}{2}(\widehat \rho^q)^2,
  n_x\xi^u\Bigr\rangle_{\partial\mathcal T_h},
   \end{split}
\end{equation*}
and rearranging the terms implies
 \begin{equation}\label{Q_xiu}
    \begin{split} 
\mathcal{Q}&(u,q;u_h,q_h;-\rho^q,-\xi^u)
  = (u\rho^q,\partial_x\rho^q)_{\mathcal T_h}+(\eta^u\rho^q,\partial_x\rho^q)_{\mathcal T_h}+ \Bigl\langle\frac{1}{2}(u+\eta^u)\bigl(\widehat\rho^q +\rho^q\bigr), n_x(\widehat\rho^q-\rho^q)\Bigr\rangle_{\partial\mathcal T_h}\\&\quad +
  (q,\xi^u\partial_x\rho^q+\rho^q\partial_x\xi^u)_{\mathcal T_h} +\langle q\xi^u,n_x(\widehat\rho^q-\rho^q)\rangle_{\partial\mathcal T_h} - \langle q\widehat\rho^q,n_x\xi^u\rangle_{\partial\mathcal T_h} \\& \quad -(\xi^u\rho^q,\partial_x\rho^q)_{\mathcal T_h}-\bigl(\frac{1}{2}(\rho^q)^2,\partial_x\xi^u\bigr)_{\mathcal T_h} +\Bigl\langle \frac{1}{2}(\widehat \rho^q)^2,
  n_x\xi^u\Bigr\rangle_{\partial\mathcal T_h}- \Bigl\langle\frac{1}{2}\xi^u(\widehat\rho^q+\rho),n_x(\widehat\rho^q-\rho^q)\Bigr\rangle_{\partial\mathcal T_h} \\&\quad -(u\mu^q,\partial_x\rho^q)_{\mathcal{T}_h}-(q,\eta^u\partial_x\rho^q+\mu^q\partial_x\xi^u)_{\mathcal{T}_h}\\&\quad + \Bigl\langle\frac{1}{2}\bigl(-u\mu^q  +\xi^u\mu^q-\eta^u\mu^q\bigr)-q\eta^u,n_x(\widehat\rho^q-\rho^q)\Bigr\rangle_{\partial\mathcal T_h}\\&\quad  + (\xi^u\mu^q,\partial_x\rho^q)_{\mathcal{T}_h} + (\rho^q\mu^q,\partial_x\xi^u)_{\mathcal{T}_h}-\frac{1}{2}((\mu^q)^2,\partial_x\xi^u){\mathcal{T}_h} - (\mu^q\eta^u,\partial_x\rho^q)_{\mathcal{T}_h} \\&\quad +
  \langle\tau_{uqq}(\widehat\rho^q-\rho^q+\mu^q), (n_x)^2(\widehat\rho^q-\rho^q)\rangle_{\partial \mathcal T_h}
  \\
  &=: \sum_{i=1}^7 \mathcal Q_i.
    \end{split}
\end{equation}
Now, we estimate each $\mathcal Q_i's$ using integration by parts, interpolation \eqref{eq:proj-est} and inverse inequalities \eqref{invesre}, transmission condition \eqref{eq:weak_TCe_xi}, Young's inequality, and the smoothness of $u$ and $q=u_x$ with assumptions $|\omega-\omega_c| = \mathcal O(h)$ on each $K_{ij}$ and so $\|\omega-\omega_c\|_{L^{\infty}( \mathcal{T}_h)} = O(h)$ and trace inequality implies $\|\omega-\omega_c\|_{L^{\infty}(\partial \mathcal{T}_h)} = \mathcal O(h^{1/2})$ for some constant $\omega_c$ on each element $K_{ij}$, where $\omega =u,q$.
\begin{align*}
    \mathcal Q_1 & = (u,\rho^q\partial_x\rho^q)_{\mathcal T_h}+(\eta^u,\rho^q\partial_x\rho^q)_{\mathcal T_h}+ \Bigl\langle\frac{1}{2}(u+\eta^u)\bigl(\widehat\rho^q +\rho^q\bigr), n_x(\widehat\rho^q-\rho^q)\Bigr\rangle_{\partial\mathcal T_h}\\
    &= -\frac{1}{2}
    (u_x,(\rho^q)^2)_{\mathcal T_h} -\frac{1}{2}
    (\eta^u_x,(\rho^q)^2)_{\mathcal T_h} +  \Bigl\langle\frac{1}{2}(u+\eta^u), n_x(\widehat\rho^q)^2\Bigr\rangle_{\partial\mathcal T_h}
    \\
    &\leq C\|u_x\|_{L^\infty(\mathcal T_h)}\|\rho^q\|^2_{\mathcal T_h} +C\|\eta^u_x\|_{L^{\infty}(\mathcal T_h)}\|\rho^q\|^2_{\mathcal T_h}  +\Bigl\langle\frac{1}{2}(u-u_c), n_x(\widehat\rho^q)^2\Bigr\rangle_{\partial\mathcal T_h} +\Bigl\langle\frac{1}{2}\eta^u, n_x(\widehat\rho^q)^2\Bigr\rangle_{\partial\mathcal T_h} \\& \leq C \|\rho^q\|^2_{\mathcal T_h} +  Ch^{1/2} \|\widehat \rho^q\|^2_{\mathcal V} \leq C \|\rho^q\|^2_{\mathcal T_h} + Ch^{1/2} \|\widehat \rho^q-\rho^q\|_{\mathcal V}^2 + Ch^{1/2} \|\rho^q\|_{\mathcal V}^2 \\&
    \leq C \|\rho^q\|^2_{\mathcal T_h} + Ch^{1/2} \|\widehat \rho^q-\rho^q\|_{\mathcal V}^2,
\end{align*}
\begin{align*}
    \mathcal Q_2 &= (q,\xi^u\partial_x\rho^q+\rho^q\partial_x\xi^u)_{\mathcal T_h} +\langle q\xi^u,n_x(\widehat\rho^q-\rho^q)\rangle_{\partial\mathcal T_h} - \langle q\widehat\rho^q,n_x\xi^u\rangle_{\partial\mathcal T_h}\\&= -(q_x,\xi^u\rho^q)_{\mathcal T_h} +\langle q\xi^u,n_x\widehat\rho^q\rangle_{\partial\mathcal T_h} - \langle q\widehat\rho^q,n_x\xi^u\rangle_{\partial\mathcal T_h}\\
    &\leq C\|q_x\|_{L^{\infty}}(\|\xi^u\|_{\mathcal T_h}^2 + \|\rho^q\|_{\mathcal T_h}^2)\leq C \|\xi^u\|_{\mathcal T_h}^2 + C\|\rho^q\|_{\mathcal T_h}^2,
\end{align*}
\begin{align*}
    \mathcal Q_3 &= -(\xi^u\rho^q,\partial_x\rho^q)_{\mathcal T_h}-\bigl(\frac{1}{2}(\rho^q)^2,\partial_x\xi^u\bigr)_{\mathcal T_h} +\Bigl\langle \frac{1}{2}(\widehat \rho^q)^2,
  n_x\xi^u\Bigr\rangle_{\partial\mathcal T_h}- \Bigl\langle\frac{1}{2}\xi^u(\widehat\rho^q+\rho),n_x(\widehat\rho^q-\rho^q)\Bigr\rangle_{\partial\mathcal T_h}
    \\& =- \Bigl\langle \xi^u, \frac{1}{2}(\rho^q)^2n_x\Bigr\rangle_{\partial\mathcal T_h} +\Bigl\langle \frac{1}{2}(\widehat \rho^q)^2,
  n_x\xi^u\Bigr\rangle_{\partial\mathcal T_h}- \Bigl\langle\frac{1}{2}\xi^u(\widehat\rho^q+\rho),n_x(\widehat\rho^q-\rho^q)\Bigr\rangle_{\partial\mathcal T_h} =0,
\end{align*}
\begin{align*}
    \mathcal Q_4 & = -(u\mu^q,\partial_x\rho^q)_{\mathcal{T}_h}-(q,\eta^u\partial_x\rho^q+\mu^q\partial_x\xi^u)_{\mathcal{T}_h}
    \\ & \leq \|u-u_c\|_{L^\infty(\mathcal T_h)}\|\mu^q\|_{\mathcal{T}_h}\|\partial_x \rho^q\|_{\mathcal{T}_h} + \|q-q_c\|_{L^\infty(\mathcal T_h)}(\|\eta^u\|_{\mathcal{T}_h}\|\partial_x \rho^q\|_{\mathcal{T}_h}  +\|\mu^q\|_{\mathcal{T}_h}\|\partial_x \xi^u\|_{\mathcal{T}_h }) \\&\leq Ch^{2k+2} + C( \|\rho^q\|^2_{\mathcal T_h} +  \|\xi^u\|^2_{\mathcal T_h}),
\end{align*}
\begin{align*}
    \mathcal Q_5 & = \Bigl\langle\frac{1}{2}\bigl(-u\mu^q  +\xi^u\mu^q-\eta^u\mu^q\bigr)-q\eta^u,n_x(\widehat\rho^q-\rho^q)\Bigr\rangle_{\partial\mathcal T_h}\\
    &\leq C\|u-u_c\|_{L^{\infty}(\partial\mathcal T_h )}h^{-1/2}h^{1/2} \|\mu^q\|_{\partial \mathcal T_h}\|\widehat\rho^q-\rho^q\|_{ \mathcal V} + C\|\mu^q\|_{L^{\infty}(\partial\mathcal T_h )}\|\xi^u\|_{\partial \mathcal T_h}\|\widehat\rho^q-\rho^q\|_{ \mathcal V}\\&\quad +C\|\eta^u\|_{L^{\infty}(\partial\mathcal T_h )}\|\mu^q\|_{\partial \mathcal T_h}\|\widehat\rho^q-\rho^q\|_{ \mathcal V}+C\|q-q_c\|_{L^{\infty}(\partial\mathcal T_h )}h^{-1/2}h^{1/2} \|\eta^u\|_{\partial \mathcal T_h}\|\widehat\rho^q-\rho^q\|_{ \mathcal V}\\& \quad 
    -\Bigl\langle\frac{1}{2}(u_c\mu^q),n_x(\widehat\rho^q-\rho^q)\Bigr\rangle_{\partial\mathcal T_h} -\langle q_c\eta^u,n_x(\widehat\rho^q-\rho^q)\rangle_{\partial\mathcal T_h}
    \\
    &\leq C(\delta)h^{2k+1} + C\|\xi^u\|^2_{ \mathcal T_h} +  C\delta\|\widehat \rho^q-\rho^q\|_{\mathcal V}^2,
\end{align*}
\begin{align*}
    \mathcal Q_6& = (\xi^u\mu^q,\partial_x\rho^q)_{\mathcal{T}_h} + (\rho^q\mu^q,\partial_x\xi^u)_{\mathcal{T}_h}-\frac{1}{2}((\mu^q)^2,\partial_x\xi^u){\mathcal{T}_h} - (\mu^q\eta^u,\partial_x\rho^q)_{\mathcal{T}_h}\\ &\leq \|\mu^q\|_{L^{\infty}}\|\xi^u\|_{\mathcal T_h}\|\rho^q_x\|_{\mathcal T_h} +\|\mu^q\|_{L^{\infty}}\|\xi^u_x\|_{\mathcal T_h}\|\rho^q\|_{\mathcal T_h} +\|\mu^q\|_{L^{\infty}}\|\mu^q\|_{\mathcal T_h}\|\xi^u_x\|_{\mathcal T_h}\\&\qquad +\|\mu^q\|_{L^{\infty}}\|\eta^u\|_{\mathcal T_h}\|\rho^q_x\|_{\mathcal T_h}\\& \leq Ch^{2k+2} + C\|\rho^q\|^2_{\mathcal T_h} + C\|\xi^u\|_{\mathcal T_h}^2,
\end{align*}
\begin{align*}
   \mathcal Q_7 = \langle\tau_{uqq}(\widehat\rho^q-\rho^q+\mu^q), (\widehat\rho^q-\rho^q)\rangle_{\partial \mathcal T_h} &= \langle\tau_{uqq}, (n_x)^2(\widehat\rho^q-\rho^q)^2\rangle_{\partial \mathcal T_h} +\langle\tau_{uqq}\mu^q,(n_x)^2 (\widehat\rho^q-\rho^q)\rangle_{\partial \mathcal T_h}\\&\leq \langle\tau_{uqq}, (n_x)^2(\widehat\rho^q-\rho^q)^2\rangle_{\partial \mathcal T_h} +Ch^{2k+1}+C\|\widehat \rho^q-\rho^q\|_{\mathcal V}^2.
\end{align*}
Putting all the estimates for $\mathcal Q_i's$ in the equation \eqref{Q_xiu}, we have
\begin{equation*}
\begin{split}
     \mathcal{Q}(u,q;u_h,q_h;-\rho^q,-\xi^u)&\leq \langle\tau_{uqq}, (n_x)^2(\widehat\rho^q-\rho^q)^2\rangle_{\partial \mathcal T_h}+ Ch^{2k+1} + C (\|\rho^q\|^2_{\mathcal T_h} + \|\xi^u\|_{\mathcal T_h}^2) \\&\qquad +C\delta \|\widehat \rho^q-\rho^q\|_{\mathcal V}^2 \\& \leq  Ch^{2k+1} + C(\|\rho^q\|^2_{\mathcal T_h} + \|\xi^u\|_{\mathcal T_h}^2),
\end{split}
\end{equation*} 
where positive constants $C$ and $\delta$, which comes from the application of Young's inequality,  are generic and independent of $h$. We have taken $\delta$ small enough so that $-\tau_{uqq}\geq C\delta$ even for small negative $\tau_{uqq}$ as in Assumption \ref{globalass}. This completes the proof.
\end{proof}

\begin{lemma}\label{H_lemma}
    Let $\mathcal{H}$ be defined by \eqref{H_ij} and assume that all the hypotheses of the Theorem \ref{thm:error_semi} hold. Then we have the following estimate
   \begin{equation}\label{H_ij_est}
    \begin{split}
        \mathcal{H}(f;u,u_h,-\xi^u)\leq -\bar C\|\widehat\xi^u - \xi^u\|^2_{  \mathcal V}+Ch^{2k+1} + C\|\xi^u\|_{ \mathcal T_h}^2-\langle (f(u)-\widehat{f(u_h)}),n_x\widehat \xi^u\rangle_{\partial\mathcal T_h}.
    \end{split}
\end{equation}
\end{lemma}
\begin{proof}
   Using the transmission condition \eqref{eq:weak_TCe_xi} and smoothness of $f$, we have
\begin{equation}\label{H_ij1}
    \begin{split}
        \mathcal{H}(f;u,u_h,-\xi^u) +\langle (f(u)-\widehat{f(u_h)})&   ,n_x\widehat \xi^u\rangle_{\partial\mathcal T_h}= (f(u)-f(u_h),\partial_x\xi^u)_{\mathcal T_h}-\langle f(u)-\widehat{f(u_h)},n_x\xi^u\rangle_{\partial\mathcal T_h}\\& \qquad +\langle (f(u)-\widehat{f(u_h)}),n_x\widehat \xi^u\rangle_{\partial\mathcal T_h}\\
        & = (f(u)-f(u_h),\partial_x\xi^u)_{\mathcal T_h}+\langle f(u)-f(u_h),n_x(\widehat \xi^u -\xi^u)\rangle_{\partial\mathcal T_h} \\&\qquad + \langle \tau_f (\widehat u_h - u_h),(n_x)^2(\widehat \xi^u-\xi^u)\rangle_{\partial\mathcal T_h}\\
        & = (f(u)-f(u_h),\partial_x\xi^u)_{\mathcal T_h}+\langle f(u)-f(u_h),n_x(\widehat \xi^u -\xi^u)\rangle_{\partial\mathcal T_h} \\&\qquad- \langle \tau_f (\widehat \xi^u - \xi^u +\eta^u),(n_x)^2(\widehat \xi^u-\xi^u)\rangle_{\partial\mathcal T_h}.
    \end{split}
\end{equation}
We expand $f(u)-f(u_h)$  about $u$ as follows
\begin{align*}
  f(u)-f(u_h)= f'(u)(u-u_h)-\frac12\,f''_{\!u}\,(u-u_h)^2
   =  f'(u)(\xi^u - \eta^u)-\frac12\,f''_{\!u}\,(\xi^u - \eta^u)^2.
\end{align*}
Using the above expansion, we get
\begin{equation}\label{H_ij2}
    \begin{split}
        \mathcal{H}(f;u,u_h,-\xi^u)& +  \langle (f(u)-\widehat{f(u_h)}),n_x\widehat \xi^u
             \rangle_{\partial \mathcal T_h}
         \\&= \underbrace{\bigl(f'(u)(\xi^u - \eta^u),\partial_x\xi^u\bigr)_{\mathcal T_h}  + \langle f'(u) (\xi^u-\eta^u),n_x(\widehat \xi^u -\xi^u)\rangle_{\partial\mathcal T_h}}_{H_1}\\&\quad \underbrace{-\Bigl(\frac12\,f''_{\!u}\,(\xi^u - \eta^u)^2,\partial_x\xi^u\Bigr)_{\mathcal T_h}-\Big\langle \frac12\, f''_{\!u}\,
  ( \xi^u-\eta^u)^2,n_x(\widehat \xi^u -\xi^u)\Big\rangle_{\partial\mathcal T_h}}_{H_2}\\&\quad \underbrace{-\langle \tau_f (\widehat \xi^u - \xi^u +\eta^u),(n_x)^2(\widehat \xi^u-\xi^u)\rangle_{\partial\mathcal T_h}}_{H_3}.
    \end{split}
\end{equation}
Incorporating interpolation \eqref{eq:proj-est} and inverse inequalities \eqref{invesre}, integration by parts, Young's inequality, and Assumption \ref{Aprioriass},  with  $\|u-u_c\|_{L^{\infty}(\mathcal{T}_h)} = \mathcal O(h)$ and  $\|u-u_c\|_{L^{\infty}(\partial \mathcal T_h)} =  \mathcal{O}(h^{1/2})$, we have 
\begin{equation*}
    \begin{aligned}
        H_1 &= \bigl(f'(u)(\xi^u - \eta^u),\partial_x\xi^u\bigr)_{\mathcal T_h}  + \langle f'(u) (\xi^u-\eta^u),n_x(\widehat \xi^u -\xi^u)\rangle_{\partial\mathcal T_h} \\
        &= \Bigl(f'(u),\frac{1}{2}\partial_x(\xi^u)^2\Bigr)_{\mathcal T_h} -\bigl(f'(u) \eta^u,\partial_x\xi^u\bigr)_{\mathcal T_h} + \langle f'(u) (\xi^u-\eta^u),n_x(\widehat \xi^u -\xi^u)\rangle_{\partial\mathcal T_h}\\& = 
        -\Bigl( f''(u)u_x,\frac{1}{2}(\xi^u)^2\Bigr)_{\mathcal T_h} + \Bigl\langle (f'(u)-f'(u_c)),n_x\frac{1}{2}(\xi^u)^2\Bigr\rangle_{\partial\mathcal T_h} +\Bigl\langle f'(u_c),n_x\frac{1}{2}(\xi^u)^2\Bigr\rangle_{\partial\mathcal T_h} \\& \quad -\bigl( (f'(u)-f'(u_c))\eta^u,\partial_x\xi^u\bigr)_{\mathcal T_h}  + \langle (f'(u)-f'(u_c))( \xi^u-\eta^u),n_x(\widehat \xi^u -\xi^u)\rangle_{\partial\mathcal T_h} \\&\quad+ \langle f'(u_c)( \xi^u-\eta^u),n_x(\widehat \xi^u -\xi^u)\rangle_{\partial\mathcal T_h}
        - \Bigl\langle f'(u_c),n_x\frac{1}{2}(\widehat \xi^u)^2\Bigr\rangle_{\partial\mathcal T_h}
        \\& \leq C\|f''(u)\|_{L^{\infty}(\mathcal{T}_h)}\|u_x\|_{L^\infty(\mathcal T_h)}\|\xi^u\|_{\mathcal{T}_h}^2+ Ch^{-\frac{1}{2}}\|f'(u)-f'(u_c)\|_{L^{\infty}(\partial\mathcal T_h)}\|\xi^u\|_{\mathcal{T}_h}^2 \\& \quad 
        -\Bigl\langle f'(u_c),n_x\frac{1}{2}(\xi^u)^2\Bigr\rangle_{\partial\mathcal T_h}
        - \Bigl\langle f'(u_c),n_x\frac{1}{2}(\widehat \xi^u)^2\Bigr\rangle_{\partial\mathcal T_h} + \langle f'(u_c) \xi^u,n_x\widehat \xi^u \rangle_{\partial\mathcal T_h}\\& \quad + h^{-1}\|f'(u)-f'(u_c)\|_{L^{\infty}(\mathcal T_h)}\|\eta^u\|_{\mathcal T_h}\|\xi^u\|_{\mathcal T_h}\\&\quad + h^{-\frac{1}{2}}\|f'(u)-f'(u_c)\|_{L^{\infty} (\partial \mathcal T_h)}(\|\eta^u\|_{\mathcal T_h}+\|\xi^u\|_{\mathcal T_h})\|\widehat\xi^u - \xi^u\|_{\mathcal V}  - \langle f'(u_c)\eta^u,n_x(\widehat \xi^u -\xi^u)\rangle_{\partial\mathcal T_h}
        \\ & 
       \leq C\|\xi^u\|_{\mathcal{T}_h}^2 - \frac{1}{2}\langle f'(u_c),n_x(\widehat\xi^u - \xi^u)^2\rangle_{\partial \mathcal T_h} + Ch^{2k+1} +C \|\widehat\xi^u - \xi^u\|^2_{  \mathcal V},
       \end{aligned}
       \end{equation*}
       
\begin{equation*}
     \begin{aligned}
       H_2 &= -\bigl(\frac12\,f''_{\!u}\,(\xi^u - \eta^u)^2,\partial_x\xi^u\bigr)_{\mathcal T_h}-\langle \frac12\, f''_{\!u}\,
  ( \xi^u-\eta^u)^2,n_x(\widehat \xi^u -\xi^u)\rangle_{\partial\mathcal T_h} \\
       &\leq C h^{-1}\|u-u_h\|_{L^\infty(\mathcal T_h)}\|\xi^u-\eta^u\|_{\mathcal{T}_h}\|\xi^u\|_{\mathcal T_h}+ C h^{-\frac{1}{2}}\|u-u_h\|_{L^\infty(\partial\mathcal T_h)}\|\xi^u-\eta^u\|_{\mathcal{T}_h}\|\widehat \xi^u-\xi^u\|_{ \mathcal V} \\
       &\leq Ch^{2k+2} + C\|\xi^u\|_{ \mathcal T_h}^2+ C\|\widehat\xi^u - \xi^u\|^2_{ \mathcal V}, 
       \end{aligned}
\end{equation*}
       \begin{equation*}
     \begin{aligned}
       H_3 &= -\langle \tau_f (\widehat \xi^u - \xi^u +\eta^u),(n_x)^2(\widehat \xi^u-\xi^u)\rangle_{\partial\mathcal T_h}\\&\leq -\langle \tau_f ,(n_x)^2(\widehat \xi^u-\xi^u)^2\rangle_{ \partial \mathcal T_h} - \langle \tau_f \eta^u,(n_x)^2(\widehat \xi^u-\xi^u)\rangle_{\partial\mathcal T_h}.
    \end{aligned}
\end{equation*}
Hence, substituting the values of $H_i's$ into the equation \eqref{H_ij1}, we have 
\begin{equation*}
    \begin{split}
        \mathcal{H}(f;u,u_h,-\xi^u) &+  \langle (f(u)-\widehat{f(u_h)}),n_x\widehat \xi^u
             \rangle_{\partial \mathcal T_h}\\&\leq - \Big\langle \tau_f - \frac{1}{2}|f'(u_c)|, (n_x)^2(\widehat \xi^u-\xi^u)^2\Big\rangle_{\partial\mathcal T_h} +Ch^{2k+1} + C\|\xi^u\|_{ \mathcal T_h}^2+ C\delta\|\widehat\xi^u - \xi^u\|^2_{ \mathcal V}.
    \end{split}
\end{equation*}
From Assumption \ref{globalass}, $C_l$ is large enough so that $-(\tau_f -\frac{1}{2}|f'(u_c)|) \leq -C_l$, and $-C_l+ C\delta \leq -\bar C$, where $\delta$ comes from the application of the Young's inequality and the generic constants $C$ and $\bar C$ are positive. This completes the proof.
\end{proof}

\end{document}